\documentclass[twoside,11pt]{article}

\usepackage{amsmath}
\usepackage{amssymb}
\usepackage{latexsym,amsthm}
\usepackage{mathtools}
\usepackage[final]{showkeys}
\usepackage{graphicx}
\usepackage{mathrsfs}
\usepackage{xcolor}
\usepackage{hyperref}
\usepackage{enumitem}
\usepackage{esint}

\newtheorem{lemma}{Lemma}
\newtheorem{corollary}{Corollary}
\newtheorem{theorem}{Theorem}

\theoremstyle{definition}
\newtheorem{remark}{Remark}

\newtheorem{example}{Example}

\title{
A biharmonic analogue of the Alt-Caffarelli problem
}

\author{Hans-Christoph Grunau
\thanks{Fakult\"{a}t f\"{u}r Mathematik, 
Otto-von-Guericke-Universit\"{a}t,
Postfach 4120,
39016 Magdeburg, Germany,
E-mail: \texttt{hans-christoph.grunau@ovgu.de}}
\and 
Marius M\"uller
\thanks{Institut f\"ur Mathematik, Universit\"at Augsburg,
	Universit\"atsstra{\ss}e 2,
	86159 Augsburg, Germany,
E-mail: \texttt{marius1.mueller@uni-a.de}}
}

\date{\today}

\begin{document}

\maketitle

\begin{abstract}
We study a natural biharmonic analogue of the classical Alt-Caffarelli problem, both under Dirichlet and under 
Navier boundary conditions. We show existence, basic properties and $C^{1,\alpha}$-regularity of minimisers. For the Navier problem we also obtain a symmetry result in case that the boundary data are radial. We find this  remarkable because the problem under investigation is of higher order. Computing radial minimisers explicitly we find that the obtained regularity is optimal.\\
\emph{AMS 2020 Subject Classification.} Primary 35R35, 31B30, Secondary 35B07, 28A12. \\ 
\emph{Keywords}. Alt-Caffarelli problem, Biharmonic operator, Free boundary, Symmetrisation.
\emph{Conflict of Interest Statement.} On behalf of all authors, the corresponding author states that there is no conflict of interest.
\end{abstract}

\section{Introduction}

For a given bounded and suitably smooth domain $\Omega\subset \mathbb{R}^n$ and a given boundary datum $\varphi\in H^2(\Omega)$, Dipierro, Karakhanyan, and Valdinoci studied in \cite{DiKaVa}
existence, regularity and further qualitative properties of minimisers of 
the following functional
\begin{equation}\label{eq:1.1}
J(u)=\int_\Omega \left( |\Delta u(x)|^2+1_{\{u>0\}}(x)\right)\, \mathrm{d}x
\end{equation}
for
\begin{equation}\label{eq:1.2}
u\in {\mathscr N} := \{v\in H^2(\Omega): (v-\varphi)\in H^1_0(\Omega)\}.
\end{equation} 
The set ${\mathscr N}$ models the so-called Navier boundary conditions where only the height of $u$ is prescribed on $\partial\Omega$ and not also the slope.
The functional $J$ in \eqref{eq:1.1} resembles a lot the Alt-Caffarelli-functional in \cite{AltCaffarelli} where $|\Delta u|^2$ has to be replaced by $|\nabla u|^2$.
However, since Alt and Caffarelli considered only positive boundary data,
thanks to the maximum principle 
their original problem remains the same when replacing $1_{\{u>0\}}(x)$ by $1_{\{u\not=0\} }(x)$.
 In fact, for large enough domains $\Omega$ minimisers in \cite{AltCaffarelli}
 exhibit large \emph{flat regions} with $u(x)=0$. 
These flat regions have a physical meaning.  They describe e.g. the \emph{shadow zone} or \emph{wake} which a jet of a fluid leaves after hitting an obstacle, cf. \cite{Tepper}.  Moreover, this  problem is related to optimisation problems involving \emph{capacities}, cf. \cite[Chapter 14]{Flucher}.

As for the functional $J(\,.\,)$ in \eqref{eq:1.1}, introduced by  Dipierro, Karakhanyan, and Valdinoci the situation is completely different.
The second author proved in  \cite{Mueller} that minimisers of $J$ have no flat parts at all (if $n=2$), since for each minimiser  $\nabla u \neq 0$ is satisfied on $\{ u = 0 \}$.   

In order to retrieve minimisers with flat regions we propose to look at a modified functional, namely
\begin{equation}\label{eq:1.3}
F(u)=\int_\Omega \left( |\Delta u(x)|^2+1_{\{u\not=0\}}(x)\right)\, \mathrm{d}x
\quad \; \;   \mbox{for} \quad  u\in {\mathscr N}\quad \mbox{or alternatively for }\quad u\in {\mathscr D},
\end{equation}
where the set
\begin{equation}\label{eq:1.4}
{\mathscr D} := \{v\in H^2(\Omega): (v-\varphi)\in H^2_0(\Omega)\}
\end{equation}
models the so-called Dirichlet boundary conditions, where  the height of $u$ 
as well as its slope are prescribed on $\partial\Omega$. 

We think that the functional $F$ exhibits minimisers which
resemble more those obtained by Alt and Caffarelli, as flatness is rewarded. 
Example~\ref{ex:1.1} and the examples from Section~\ref{sec:radial} give evidence to this statement. 

Nevertheless one should have in mind that the variational problem is of second order, i.e. that any admissible function and in particular any minimiser has to be in $H^2$ which means that on the boundary of the ``flat set'' $\{x\in \Omega: u(x)=0\}$ (or, more precisely, its interior), 
not only the values of the function but also of its derivatives have to match somehow. Outside this flat set, minimisers turn out to be biharmonic, and for biharmonic functions no strong  maximum principle and, in general, not even a weak form of a comparison principle is available. In the variational context this is reflected by the fact that for $u\in H^2$, in general $u^+\not\in H^2$, i.e. the ``Stampacchia trick'' (comparing the energy of $u$ and $u^+$)  {does not apply}. 
For features like positivity of superbiharmonic functions in balls, sign change in general domains, but also dominance of positivity in general domains, one may give a look at \cite[Chapters 1, 2.6, 5 and 6]{GGS}. Having these features in mind one may expect that in the variational problem under investigation, minimisers for positive boundary data may look ``almost positive'' but may at the same time exhibit also  oscillations, in particular close to the ``flat set''. Although a detailed investigation of these qualitative properties of minimisers would be in our opinion very interesting, this is beyond the scope of the present work and may be addressed in future research, in particular in combination with numerical approximations.

In the present paper we study existence, some basic qualitative properties and (optimal) regularity of minimisers in $\mathscr{D}$ and $\mathscr{N}$. 
Next we summarise the main findings. 
\begin{flushleft}
    \textbf{I. Interior $C^{1,\alpha}$-regularity, $\alpha \in (0,1)$.} 
\end{flushleft}

Adapting the methods in \cite{DiKaVa} we infer first that minimisers lie in $C^{1,\alpha}(\Omega)$ for all $\alpha \in (0,1)$, cf. Section \ref{sec:BMO}. We also show that the set $\{ u \neq 0 \}$ is in general nonempty and may have large measure, cf. Section \ref{sec:flatset}. 

 We show further that  minimisers $u\in \mathscr{N}$ with 
  radial boundary data in $\Omega = B_1(0)$ 
 are always smooth in a neighborhood of the boundary
 and  radially symmetric.
\begin{flushleft}
    \textbf{II. Boundary regularity (for radial data).} 
\end{flushleft}
 The proof of the boundary regularity result is given in Section~\ref{sec:app_BMO} and relies on a ``Kelvin transform''-like reflection trick applied to $\Delta u$.
 It is by no means obvious how to extend this proof and how to keep the free boundary of $\{x\in \Omega:u(x)=0\}$ apart from the boundary $\partial \Omega$, when  general smooth domains and general smooth boundary data are considered. We  have to leave this as a challenging open problem. Only in dimensions $n\in\{1,2,3\}$ such a reasoning is obvious thanks to Sobolev embedding.  \newpage
\begin{flushleft}
    \textbf{III. Radial symmetry.} 
\end{flushleft}
 
  The proof of the radial symmetry result is based upon a modification of \emph{Talenti's symmetrisation method} (\cite{Talenti}) in annular shaped domains, applied to $|\Delta u|$, cf. Section \ref{sec:naviersym} and {Appendix}~\ref{app:rearrangement}. We find this  result remarkable, because 
 for higher order problems  symmetry results are in general rather involved and their validity is often limited due to the abovementioned lack of  maximum and comparison  principles, see e.g. \cite[Section~7.1]{GGS}. 
 At the same time we have to leave the question open whether  the same radiality result holds  for minimisers in $\mathscr{D}$. 

\begin{flushleft}
    \textbf{IV. Impossibility of (general) $C^2$-regularity.}
\end{flushleft}

 We finally obtain that minimisers in $\mathscr{N}$ do in general not lie in $C^2(\Omega)$, which shows that the obtained $C^{1, \alpha}$-regularity result is optimal. According to the abovementioned symmetry result, Navier  minimisers $u$ in balls for constant boundary data $\varphi(x)\equiv \operatorname{const}=u_0$ are always radial.
 In $B_1(0)\subset \mathbb{R}^2$ we compute these minimisers explicitly.  We see that for $u_0>0$ small enough they exhibit a flat region $\{x\in B_1(0):u(x)=0\}\not=\emptyset$ and we infer that they do not lie in $C^2(\Omega)$, cf. Section \ref{sec:radial}. See Corollary~\ref{cor:opt_reg} in Section~\ref{sec:naviersym}.
\\
\\
One should remark that the regularity results I and IV leave open whether $C^{1,1}$-regularity can generally be obtained. Due to the absence of a PDE for the minimizer, this regularity discussion goes beyond the scope of this article. Notice however that all the radial minimisers we study in Section \ref{sec:radial} are $C^{1,1}$-regular. In \cite{DiKaVa2019}, Dipierro, Karakhanyan and Valdinoci study a singular perturbation problem for the functional $J$ in \eqref{eq:1.1}. It is observed that the $C^{1,1}$-norm of the constructed approximate solutions may degenerate, cf. \cite[Theorem 1.7]{DiKaVa2019}. This is a remarkable phenomenon. Notice however that it does not imply that solutions are generally not $C^{1,1}$-regular. Indeed, the second author proves that (if $n=2$) minimisers of $J$ even have higher $C^{2,1}$-regularity, cf. \cite[Theorem 5.2]{Mueller2023}. 
We think that the discussion of $C^{1,1}$-regularity is an interesting, yet difficult, question. 
 

\section{Existence, basic properties, and regularity of minimisers}
In this section we prove the existence and basic properties of minimisers in $\mathscr{D}$ and $\mathscr{N}$. We formulate the regularity result, which will be proved in Section~\ref{sec:BMO} and whose optimality is shown in  Corollary~\ref{cor:opt_reg} at the end of Section~\ref{sec:naviersym}.

In one dimension minimisers can be computed explicitly. Doing so, we see that the flat set $\{ u = 0 \}$ may become arbitrarily large, but it may also become void. Moreover we study the value of the infimum of the energy.  

\begin{theorem}[Existence of minimisers]\label{thm:existence}
	Assume that $\Omega\subset\mathbb{R}^n$ is a bounded $C^2$-smooth domain
	and that $\varphi\in H^2(\Omega)$.
	The functional $F$ attains its infimum on  ${\mathscr N}$ as well as  on  ${\mathscr D}$ which are defined in \eqref{eq:1.2} and \eqref{eq:1.4}.	
\end{theorem}

\begin{proof}
	We give the proof only on ${\mathscr N}$, because on ${\mathscr D}$ the proof is similar and as for the $H^2$-boundedness even simpler. Since $\Omega$ is bounded the infimum of $F$ on ${\mathscr N}$ is finite, i.e.
	$$
	\alpha:= \inf_{v\in {\mathscr N}} F(v) \in [0,\infty).
	$$
	We consider a minimising sequence
	$$
	(v_k)_{k \in \mathbb{N}}\subset {\mathscr N}: \quad \lim_{k\to\infty} F(v_k)=\alpha.
	$$
	In particular, $(\| \Delta v_k\|_{L^2(\Omega)})_{k\in\mathbb{N}}$ is bounded.
	Thanks to elliptic $H^2$-theory \cite[Chapter 9]{GilbargTrudinger} and the $C^2$-smoothness and boundedness of $\Omega$ we see that also
	$(\|  v_k\|_{H^2(\Omega)})_{k\in\mathbb{N}}$ is bounded. Hence we find a $u\in H^2(\Omega)$ so that after passing to a subsequence we have
	$$
	v_k \rightharpoonup u \mbox{\ in\ } H^2(\Omega),\quad 
	v_k \to u \mbox{\ in\ } H^1(\Omega),\quad 
	v_k(x) \to u(x) \mbox{\  for a.e.\ }x \in\Omega.
	$$
	This yields immediately that also $u-\varphi\in H^1_0(\Omega)$, hence
	$
	u\in {\mathscr N}.
	$
	(When working on ${\mathscr D}$, one exploits that $\varphi+H^2_0(\Omega)$ is convex and closed and so, weakly sequentially closed in $H^2(\Omega)$. This shows that $u\in \varphi+H^2_0 (\Omega)$.) As it is well known, the $H^2$-norm is weakly sequentially {lower semicontinuous} in $H^2$, in particular
	$$
	\int_\Omega |\Delta u(x)|^2\, \mathrm{d}x\le \liminf_{k\to\infty}\int_\Omega |\Delta v_k(x)|^2\, \mathrm{d}x.
	$$
	Moreover, 	$v_k(x) \to u(x) \mbox{\ almost everywhere in\ }\Omega$ implies
	 that $1_{\{u\not=0\}}(x)\le \liminf_{k\to\infty}1_{\{v_k\not=0\}}(x)$
	almost everywhere in $\Omega$. Fatou's lemma then yields
	$$
	\int_\Omega 1_{\{u\not=0\}}(x)\, \mathrm{d}x
	\le \int_\Omega\liminf_{k\to\infty} 1_{\{v_k\not=0\}}(x)\, \mathrm{d}x
	\le \liminf_{k\to\infty}\int_\Omega 1_{\{v_k\not=0\}}(x)\, \mathrm{d}x.
	$$
	 All in all we see that
	\begin{align*}
	\alpha\le F(u) =&\int_\Omega |\Delta u(x)|^2\, \mathrm{d}x+\int_\Omega 1_{\{u\not=0\}}(x)\, \mathrm{d}x
	\le \liminf_{k\to\infty}\int_\Omega |\Delta v_k(x)|^2\, \mathrm{d}x
	+\liminf_{k\to\infty}\int_\Omega 1_{\{v_k\not=0\}}(x)\, \mathrm{d}x\\
	\le &\liminf_{k\to\infty}\left(\int_\Omega |\Delta v_k(x)|^2\, \mathrm{d}x
	+\int_\Omega 1_{\{v_k\not=0\}}(x)\, \mathrm{d}x\right)
	=\liminf_{k\to\infty} F(v_k)=\alpha.
	\end{align*}
	This proves that $u$ minimises $F$ in ${\mathscr N}$.
\end{proof}
 
\begin{remark}\label{rem:radial}
	In case that $\Omega=B_R(0)$ is a ball of radius $R$ and the boundary datum $\varphi$ is radially symmetric, one may also consider the smaller sets ${\mathscr N}_{\operatorname{rad}}:=\{v\in {\mathscr N}: v \mbox{\ is radially symmetric}\}$ and ${\mathscr D}_{\operatorname{rad}}:=\{v\in {\mathscr D}: v \mbox{\ is radially symmetric}\}$. Similar to Theorem~\ref{thm:existence} one shows that $F$ 
	attains its minimum also on ${\mathscr N}_{\operatorname{rad}}$ and ${\mathscr D}_{\operatorname{rad}}$, respectively. 
	An interesting question in the context of biharmonic problems is whether we have symmetry, i.e. whether the minimisers on the radial and the general sets of admissible functions  coincide.
	We show in Section \ref{sec:naviersym} that minimisers  in $\mathscr{N}_{\operatorname{rad}}$ 
	coincide  indeed with those in $\mathscr{N}$. In $\mathscr{D}$, i.e. under Dirichlet boundary conditions, we have to leave this question open.
\end{remark}

\begin{example}\label{ex:1.1}
		In the special one-dimensional case the problem may be solved explicitly.
		We consider the interval $\Omega=(-R,R)$ of length $2R$ under Navier boundary conditions
		$$
		u(\pm R)=1, \qquad u''(\pm R)=0.
		$$
		According to Theorem~\ref{thm:symmetry}\textemdash or more elementary, according to the remark at the end of this example\textemdash any minimiser is even.  
		Since minimisers in this case have to be at least $C^1$-smooth and biharmonic on $\Omega \setminus  \{x \in \Omega: u(x) =u'(x) = 0\}$, 
		 candidates for a minimiser are either the function
		$$
		u_0(x)\equiv 1 \quad \mbox{\ in\ }\quad  (-R,R)
		$$
		(biharmonic in the whole interval and empty ``free part'') or 
		the following functions for
		$\rho \in (0,R)$:
		$$
		u_\rho (x)=
		\begin{cases}
		0, &\mbox{for}\quad |x|\in [0,\rho],\\
		 \frac{(|x|-\rho)^2(3R-2\rho-|x|)}{2(R-\rho)^3},
		&\mbox{for}\quad |x|\in [\rho,R]
		\end{cases}
		$$
		(with $[-\rho,\rho]$ as ``free part'' and biharmonic on $(\rho,R)$).
		
		As for the energies of these candidates we calculate
		$$
		F(u_0)= 2R
		$$
		and 
		$$
		F(u_\rho)= 2(R-\rho) +\frac{6}{(R-\rho)^3}.
		$$
		Admitting (formally) $\rho\in [0,R)$ we see that the optimal choice 
		$\rho_{\min}$ of $\rho$ and the corresponding energy are:
		$$
		\rho_{\min} =\begin{cases}
		R-\sqrt3 , \quad &\mbox{\ if\ } R>\sqrt3,\\
		0 , \quad &\mbox{\ if\ } 0\le R\le\sqrt3,
		\end{cases}
		\qquad 
		F(u_{\rho_{\min}})  
		=\begin{cases}
		 2\sqrt3+\frac2{\sqrt3}, \quad &\mbox{\ if\ } R>\sqrt3,\\
		 2R+\frac{6}{R^3}, \quad &\mbox{\ if\ } 0\le R\le\sqrt3.
		\end{cases}
		$$
		Comparing this with $u_0$ we see that the minimiser/s $u_{\min}$ of our functional  and its 
		energy is/are given by 
		$$
		u_{\min} =\begin{cases}
		u_0 , \quad &\mbox{\ if\ } 0\le  R\le \sqrt3+\frac{1}{\sqrt3},\\
		u_{R-\sqrt3} , \quad &\mbox{\ if\ }  R\ge\sqrt3+\frac{1}{\sqrt3},
		\end{cases}
		\qquad 
		F(u_{\min})  
		=\begin{cases}
		2R, \quad &\mbox{\ if\ } 0\le R\le\sqrt3+\frac{1}{\sqrt3},\\
		2\sqrt3+\frac2{\sqrt3}, \quad &\mbox{\ if\ } R\ge\sqrt3+\frac{1}{\sqrt3}.
		\end{cases}
		$$
				\emph{These calculations may be carried out separately for the ``left'' and the ``right'' part of the minimisers. From this one can directly infer that they must be symmetric about $0$.}
\end{example}

\begin{remark}
One thing that becomes visible is that minimisers may be not unique. In the previous example this is the case for $ R=\sqrt3+\frac{1}{\sqrt3}$. This means that minimisers do \emph{not} depend continuously on  the domain.

Another thing that becomes visible is that
the flat set $\{u = 0 \}$ in the previous example cannot be arbitrarily small, if one excludes cases where $\{ u_{\min} = 0 \} = \emptyset$. Indeed, notice that if $R\ge \sqrt{3}+\frac{1}{\sqrt3}$ the ratio 
\begin{equation*}
    \frac{|\{ u_{\min} = 0 \}|}{|\Omega|} = \frac{2(R-\sqrt{3})}{2R}=1-\frac{\sqrt3}{R}
\end{equation*}
attains precisely the values in $[\frac{1}{4},1]$.  
\end{remark}

Next we state our main regularity result for minimisers.

\begin{theorem}\label{thm:regularity_1}
	Assume that $\Omega\subset\mathbb{R}^n$ is a bounded $C^2$-smooth domain
	and that $\varphi\in H^2(\Omega)$. The minimiser $u$ of $F$ on  ${\mathscr N}$ or on $ {\mathscr D}$, respectively, lies in $C^{1,\alpha}(\Omega)$ for any $\alpha \in (0,1)$, and is smooth and biharmonic on the open set $\Omega \setminus \{x\in \Omega:u(x)= \nabla u (x)=0\}$.
\end{theorem}

We prove this result at the end of Section~\ref{sec:BMO}. We will moreover see in Section \ref{sec:radial} and Section  \ref{sec:naviersym} that  {minimisers do in general not} lie in $C^2(\Omega)$.

\begin{remark}\label{rem:1.2}
The regularity result only  studies interior regularity. However, one would  also like to conclude boundary regularity, provided that the assumptions keep the flat region $\{x\in \Omega:u(x)= \nabla u (x)=0\}$ away from the boundary. To this end we assume 
 that $\Omega$ has smooth boundary and that  $\varphi$ is smooth and strictly positive on $\overline{\Omega}$. Then it is possible to conclude further 
in the special situation of Theorem~\ref{thm:app_Delta_BMO} 
and in general
in dimensions $n=1,2,3$ where ${H^2(\Omega)} \hookrightarrow C^0(\overline{\Omega})$. 
	Indeed, since $\varphi|_{\partial \Omega}>0$ and each minimiser $u$  lies in $ C^0(\overline{\Omega})$ one has that $\{x\in \overline \Omega:u(x)>0\}$ 
	is a neighbourhood of $\partial \Omega$. In particular there exists also a smooth tubular neighbourhood $\Omega'\subset\overline{\Omega}$ 
	of $\partial\Omega$ such that $u$ is  biharmonic and smooth on $\overline{\Omega'} \setminus \partial \Omega$. 
	The smoothness of its boundary values $\varphi$ on $\partial \Omega$ yield that then also $u \in C^\infty(\overline{\Omega'})$.  
	As a consequence, the Dirichlet or Navier boundary conditions are attained in a classical sense. In the case of Navier
	boundary conditions this means that
	$$
	u|_{\partial \Omega}=\varphi|_{\partial \Omega}\quad \mbox{and}\quad \Delta u|_{\partial \Omega}=0.
	$$ 
	In the case of minimisation in $\mathscr{N}_{\operatorname{rad}}$ or $\mathscr{D}_{\operatorname{rad}}$ (cf. Remark \ref{rem:radial}) the boundary regularity can also be obtained for arbitrary dimensions, 
	since by radial Sobolev embedding $\mathscr{N}_{\operatorname{rad}}, \mathscr{D}_{\operatorname{rad}} \subset C^0(\overline{\Omega}\setminus \{0\})$. 
\end{remark}

\begin{remark}
If one sets $A := \{ x \in \Omega : u(x) = \nabla u (x) = 0 \}$ one infers that each minimiser $u \in \mathscr{D}$ (formally) solves the \emph{biharmonic Dirichlet problem}
\begin{equation*}
\begin{cases}
    \Delta^2 u = 0 & \textrm{in } \Omega \setminus A, \\
     u =\varphi,\quad  \partial_\nu u  = \partial_\nu \varphi& \textrm{on }\partial \Omega,\\
      u = |\nabla u | = 0 & \textrm{on } \partial  A ,
    \end{cases}
\end{equation*}
where $\partial_\nu$ denotes the exterior normal derivative.

In contrast to this, a minimiser $u \in \mathscr{N}$ solves a \emph{mixed Dirichlet-Navier problem}
\begin{equation*}
\begin{cases}
    \Delta^2 u = 0 & \textrm{in } \Omega \setminus A, \\ u = \varphi,\quad \Delta u = 0 & \textrm{on } \partial \Omega, \\ 
     u= |\nabla u | = 0 & \textrm{on }\partial A.
    \end{cases}
\end{equation*}
This means that in \emph{both} cases our problem does \emph{not} decouple into a second order system for $u$ and $\Delta u$.
\end{remark}


\subsection{The energy infimum}

Another noteworthy thing is that in the case of the Navier problem, the energy infimum is bounded above independently of the boundary datum, while in the case of the Dirichlet problem this is not the case.

\begin{lemma}[The energy infimum] \label{lem:energinfi}
Let $\Omega \subset \mathbb{R}^n$ be a domain with $C^2$-smooth boundary and $\varphi \in H^2(\Omega)$. Then 
\begin{equation}\label{eq:Navierinf}
    \inf_{\psi \in \mathscr{N}} F(\psi) \leq |\Omega|
\end{equation}
and 
\begin{equation}\label{eq:Dirinf}
    \inf_{\psi \in \mathscr{D}} F(\psi) \geq \frac{1}{|\Omega|} \left(\int_{\partial \Omega} \nabla \varphi \cdot \nu_\Omega \; \mathrm{d}\mathcal{H}^{n-1} \right)^2.
\end{equation}
Here $\mathcal{H}^{n-1}$ denotes the $n-1$-dimensional Hausdorff measure.
\end{lemma}
\begin{proof}
 We start with \eqref{eq:Navierinf}. Since the domain is suitably smooth and $\varphi \in H^2(\Omega)$ the unique weak solution $\psi_0 \in H^1(\Omega)$ of the Dirichlet problem  
 \begin{equation*}
     \begin{cases}
         \Delta \psi_0 = 0  & \textrm{in }\Omega \\ \; \; \; \psi_0 = \varphi & \textrm{on }\partial \Omega
     \end{cases}
 \end{equation*}
 lies actually in $H^2(\Omega)$. 
 One infers that $\psi_0 \in \mathscr{N}$ by \eqref{eq:1.2}. Moreover one computes 
 \begin{equation*}
     F(\psi_0)  = \int_\Omega (\Delta \psi_0)^2 \; \mathrm{d}x + |\{ \psi_0 \neq 0\}| = |\{ \psi_0 \neq 0 \}| \leq |\Omega|,   
 \end{equation*}
 whereupon \eqref{eq:Navierinf} follows. To show  \eqref{eq:Dirinf} we fix some arbitrary $\psi \in \mathscr{D}$ and simply use the Cauchy-Schwarz inequality. 
 \begin{equation*}
     F(\psi) \geq \int_\Omega ( \Delta \psi)^2 \; \mathrm{d}x \geq \frac{1}{|\Omega|} \left( \int_\Omega \Delta \psi \; \mathrm{d}x \right)^2 = \frac{1}{|\Omega|} \left( \int_{\partial \Omega}
     \nabla \psi \cdot \nu_\Omega \; \mathrm{d}\mathcal{H}^{n-1} \right)^2 . 
 \end{equation*}
 Using that $\nabla \psi = \nabla \varphi$ on $\partial \Omega$ (in the sense of Sobolev traces) we obtain the claim. 
\end{proof}

\begin{remark}
We see that in $\mathscr{D}$ arbitrarily small boundary functions can lead to minimisers of arbitrarily high Dirichlet energy. Indeed, for $\varepsilon> 0$ and $k \in \mathbb{N}$ we consider $\varphi_{\varepsilon,k}\in H^2 (\Omega)$ chosen in such a way that $\varphi_{\varepsilon,k} = \varepsilon + k d_\Omega$ in a fixed suitable neighbourhood of $\partial \Omega$, 
where $d_\Omega$ is the \emph{signed distance function} of $\Omega$, cf. \cite[Appendix p. 381]{GilbargTrudinger}. Then $\varphi\vert_{\partial \Omega} = \varepsilon$ and $\nabla \varphi \cdot \nu_\Omega \vert_{\partial \Omega} = k$, cf. \cite[Lemma 14.16]{GilbargTrudinger}. 
By \eqref{eq:Dirinf} we infer that
\begin{equation*}
    \inf_{\psi \in \mathscr{D}} F(\psi) \geq  \frac{1}{|\Omega|}k^2 \mathcal{H}^{n-1}(\partial \Omega)^2. 
\end{equation*}
For $k$ large and $\varepsilon$ small we have produced in the $C^0$-sense small boundary data with  large optimal energies.
\end{remark}

\begin{remark}
For the minimisation in $\mathscr{N}$ we actually obtain a \emph{dichotomy result}. If $\inf_{\psi \in \mathscr{N}} F(\psi) = |\Omega|$ then one immediately can obtain a harmonic minimiser with empty free boundary, as the proof of Lemma  \ref{lem:energinfi} reveals. If $\inf_{\psi \in \mathscr{N}} F(\psi) < |\Omega|$ we infer that $|\{u = 0 \}| > 0$, i.e. the free boundary is nonempty.   
\end{remark}


\subsection{Nonemptiness of the ``flat set'' and free boundary}\label{sec:flatset}

We show that when the set $\Omega$ is large enough 
compared to the boundary conditions, then any minimiser 
$u\in \mathscr{N}$ or $u\in \mathscr{D}$ exhibits a nonempty flat set $\{x\in\Omega:u(x)=0\}$. For brevity we discuss only the case of Dirichlet conditions;  Navier conditions can be treated analogously.
More precisely we have the following result which holds irrespective of the  shape of domains.
\begin{theorem}\label{thm:flatsetnonempty}
Let $B_2(0)\subset \Omega\subset\mathbb{R}^n$ be any given $C^2$-smooth  
domain and $\varphi \in H^2(\Omega)$ any 
boundary condition. 
 Let $e_n$ denote the volume of the unit ball $B_1(0)\subset \mathbb{R}^n$. Then there exists a constant $C_1(\Omega,\varphi)$ such that for any $R\in (\sqrt[4]{C_1/e_n},\infty)$ the following holds.
If we set $\Omega_R:= R\, \cdot\, \Omega$ and
$$
\varphi_R\in H^2(\Omega_R),\quad \varphi_R(x):=\varphi (x/R),\qquad 
\mathscr{D}_R:=\{v\in H^2(\Omega_R): (v-\varphi_R)\in H^2_0(\Omega_R)\},
$$
then for any minimiser $u\in \mathscr{D}_R$ of $F$ on $\Omega_R$ 
the ``flat set'' $\{x\in \Omega_R: u(x)=0\}$ has positive measure.
More precisely,
$$
|\{x\in \Omega_R: u(x)=0\}| \ge (e_nR^4-C_1)\cdot R^{n-4}.
$$
\end{theorem}

\begin{proof}
	We consider a function $v_1\in \mathscr{D}_1$ such that
	$$
	\{x\in\Omega_1: v_1(x)=0\}\supset B_1(0).
	$$
	We define 
	$
	v_R\in \mathscr{D}_R, v_R(x):=v_1(x/R)
	$
	and calculate:
	\begin{align*}
		\inf_{v\in \mathscr{D}_R} F(v)\le F(v_R)\le&
		\frac{1}{R^4}\int_{\Omega_R} (\Delta v_1)^2 (x/R)\, \mathrm{d}x +R^n \left( |\Omega_1|-e_n\right) \\
		=&R^{n-4}\int_{\Omega_1} (\Delta v_1)^2 (x)\, \mathrm{d}x +R^n \left( |\Omega_1|-e_n\right)\\
		=& C_1 R^{n-4}+R^n \left( |\Omega_1|-e_n\right)
	\end{align*}
where
$$
C_1:= \int_{\Omega_1} (\Delta v_1)^2 (x)\, \mathrm{d}x .
$$
For the flat set of any minimiser $u \in \mathscr{D}_R$ we conclude:
\begin{align*}
|\{x\in \Omega_R: u(x)=0\}|=&|\Omega_R|- |\{x\in \Omega_R: u(x)\not=0\}|\ge R^n |\Omega_1|-  F(u)\\
\ge& R^n |\Omega_1|-F(v_R)
\ge   R^n |\Omega_1|- C_1 R^{n-4}+R^n \left(e_n- |\Omega_1|\right)\\
=&e_n R^n - C_1 R^{n-4}. \qedhere
\end{align*}
\end{proof}
 

\section{$BMO$-estimates and regularity}

\subsection{A local $BMO$-estimate for the Laplacian of minimisers}
\label{sec:BMO}

\begin{theorem}\label{thm:Delta_BMO}
	Let $\Omega\subset\mathbb{R}^n$ be a bounded  $C^2$-smooth 
	domain with \emph{maximal inscribed radius}
	$$
	r_\Omega:= \sup_{x\in \Omega} d(x,\partial \Omega)
	$$
and  $\varphi\in H^2(\Omega)$. We consider any minimiser $u$ of $F$ on  ${\mathscr N}$ or on $ {\mathscr D}$, respectively.

Then for any $R_0\in (0,\frac13 r_\Omega)$ there exists a constant $C=C(u,R_0)$ such that for all $x_0\in\Omega $ with dist$(x_0,\partial \Omega)\ge 3R_0$ and all $r\in(0,R_0)$ one has the following estimate of bounded mean oscillation type
\begin{equation}\label{eq:2.1}
	\int_{B_r (x_0)}|\Delta u -(\Delta u)_{x_0,r}|^2\, \mathrm{d}x\le Cr^n,
\end{equation}
where
$$
(\Delta u)_{x_0,r}:=\fint_{B_r(x_0)}\Delta u(x)\, \mathrm{d}x:= \frac1{|B_r(x_0)|}\int_{B_r(x_0)}\Delta u(x)\, \mathrm{d}x
$$
denotes the mean value of $\Delta u$ in $B_r(x_0)$.  In particular, $\Delta u\in BMO_{\operatorname{loc}} (\Omega)$.
\end{theorem}
For the existence of such minimisers we refer to Theorem~\ref{thm:existence}.

\begin{proof}
	This is pretty much along the lines of \cite[Theorem 1.1]{DiKaVa}. 
	However, for the reader's convenience we give a very detailed elaboration
	in Appendix~\ref{sec:app_proof_BMO}.
\end{proof}

\begin{remark}\label{rem:2}
{From} the fact that each minimiser $u$ of $F$ in ${\mathscr N}$ or $ {\mathscr D}$ satisfies $\Delta u \in BMO_{\operatorname{loc}}(\Omega)$, we deduce $\Delta u \in L^q_{\operatorname{loc}}(\Omega)$ for all $q \in [1,\infty)$, see e.g. \cite[Corollary in Chapter IV.1.3]{Stein}.
We obtain then by elliptic regularity that $u \in W^{2,q}_{\operatorname{loc}}(\Omega)$ for each $q \in [1,\infty)$. In particular, by Sobolev's embedding one has $u \in C^{1,\alpha}(\Omega)$ for each $\alpha \in (0,1)$.  
\end{remark}

With this regularity we are finally able to prove Theorem \ref{thm:regularity_1}.

\begin{proof}[Proof of Theorem \ref{thm:regularity_1}]
Let $u$ be as in the statement. That $u \in C^{1,\alpha}(\Omega)$ follows from the previous remark. 
We also notice that $\{x \in \Omega: u(x) = \nabla u(x) =0 \}$ is closed in $\Omega$ by the fact that $u \in C^1(\Omega)$. Hence the set $\Omega \setminus \{ u = \nabla u = 0\}$ is actually open. The proof of the biharmonicity on $\Omega \setminus \{ u = \nabla u = 0\}$ is divided into two steps. \\
\textbf{Step 1.} Biharmonicity on $\{ u  \neq 0 \}$. One can argue as above that $\{ u = 0\} \subset \Omega$ is closed in $\Omega$ and hence $\{ u \neq 0 \}$ is open.
Now let $\eta \in C_0^\infty(\{ u \neq 0 \})$ be arbitrary. Since by the previous remark $|u| \in C^0(\Omega)$ we infer that there exists $\delta > 0$ such that $|u| > \delta$ on $\mathrm{supp}(\eta)$. In particular for each $\varepsilon \in (-\frac{\delta}{\|\eta\|_\infty}, \frac{\delta}{\|\eta\|_\infty})$ one has $\{ u + \varepsilon \eta \neq 0 \} = \{ u \neq 0 \}$. Since $u$ is a minimiser we have $F(u) \leq F(u + \varepsilon \eta)$ for all such $\varepsilon$. Using that the measure terms in $F$ coincide we obtain 
\begin{equation}
 \int_\Omega (\Delta u)^2 \; \mathrm{d}x \leq \int_\Omega (\Delta u + \varepsilon \Delta \eta)^2 \; \mathrm{d}x, 
\end{equation} 
for all $\varepsilon \in (- \frac{\delta}{\|\eta\|_\infty}, \frac{\delta}{\|\eta\|_\infty})$. Rearranging we obtain 
\begin{equation*}
0 \leq \varepsilon \int_\Omega \Delta u \Delta \eta \; \mathrm{d}x + \varepsilon^2 \int_\Omega (\Delta \eta)^2 \; \mathrm{d}x. 
\end{equation*}
First we look at the case of $\varepsilon > 0$. Dividing by $\varepsilon$ and letting $\varepsilon \downarrow 0$ then yields 
\begin{equation*}
0 \leq \int_\Omega \Delta u \Delta \eta \; \mathrm{d}x. 
\end{equation*}
The same procedure with $\varepsilon < 0$ leads to 
\begin{equation*}
0 \geq \int_\Omega \Delta u \Delta \eta \; \mathrm{d}x. 
\end{equation*}
The previous two inequalities for arbitrary $\eta \in C_0^\infty( \{ u \neq 0 \})$  imply that $ u$ is (weakly) biharmonic on $\{u \neq 0 \}$ and hence smooth on $\{ u \neq 0 \}$. Elliptic regularity then yields that $u$ is smooth on $\{ u \neq 0 \}$ and also biharmonic on $\{ u \neq 0 \}.$ \\
\textbf{Step 2.} Biharmonicity on $\{ u = 0, \nabla u \neq 0 \}$. Let $x_0 \in \Omega$ be such that $u(x_0)= 0$ and $\nabla u(x_0) \neq 0$. We will show that there exists a neighbourhood of $x_0$ on which $u$ is biharmonic. To this end first notice that by $C^1$-regularity there exists $r > 0$ such that $\nabla u \neq 0$ on $B_r(x_0)$ and $B_r(x_0) \subset \subset \Omega$. In particular this nonvanishing gradient implies that $\{ u = 0 \} \cap B_r(x_0)$ is a $C^1$-submanifold of $B_r(x_0)$ and hence $|\{ u = 0 \} \cap B_r(x_0)| = 0$. We infer therefore $|\{ u \neq 0 \} \cap B_r(x_0)| = |B_r(x_0)|$. 
Looking at the minimisation problem of $\chi \mapsto \int_{B_r(x_0)} ( \Delta \chi )^2 \; \mathrm{d}x $ among all $\chi \in u + H_0^2(B_r(x_0))$ we infer the existence of a unique minimiser $w \in u + H_0^2(B_r(x_0))$. This minimiser $w$ additionally satisfies $\Delta^2 w = 0$ weakly in $B_r(x_0)$. 
Next define $v: \Omega \rightarrow \mathbb{R}$ to be (a.e.) 
\begin{equation*}
    v(x) := \begin{cases}
        u(x), & x \in \Omega \setminus B_r(x_0), \\ 
        w(x), & x \in B_r(x_0) .
    \end{cases}
\end{equation*}
Since the Sobolev traces of $u$ and $w$ and of their first derivatives match at $\partial B_r(x_0)$ and $B_r(x_0) \subset \subset \Omega$ we further obtain that $v \in H^2(\Omega)$ and has the same boundary values as $u$ (in both $\mathscr{D}$ and $\mathscr{N}$). Therefore $v$ is admissible for our minimisation problem. In particular $F(u) \leq F(v)$. Using the definition of $v$ we find 
\begin{align}
    0 & \leq F(v) - F(u) = \int_\Omega (\Delta v)^2 \; \mathrm{d}x - \int_\Omega( \Delta u)^2 \; \mathrm{d}x + |\{ v \neq 0 \}| - |\{ u \neq 0 \}| 
     \label{eq:190} \\ 
     & = \int_{B_r(x_0)} ( \Delta w)^2 \; \mathrm{d}x - \int_{B_r(x_0)} (\Delta u)^2 \; \mathrm{d}x + |\{ w \neq 0 \} \cap B_r(x_0) | - |\{ u \neq 0 \} \cap B_r(x_0) \}|. \nonumber
\end{align}
Now the definition of $w$ yields that \begin{equation}\label{eq:gleichheit}
    \int_{B_r(x_0)} ( \Delta w)^2 \; \mathrm{d}x - \int_{B_r(x_0)} (\Delta u)^2 \; \mathrm{d}x \leq  0.
\end{equation}
Moreover as we have discussed above one has $|\{ u \neq 0 \} \cap B_r(x_0)| = |B_r(x_0)|$ and hence also 
\begin{equation*}
    |\{ w \neq 0 \} \cap B_r(x_0) | - |\{ u \neq 0 \} \cap B_r(x_0) | = |\{ w \neq 0 \} \cap B_r(x_0) | - |B_r(x_0) | \leq 0.
\end{equation*}
The previous two inequalities together with \eqref{eq:190} yield $0 \leq F(v)- F(u) \leq 0$ and hence $F(u) = F(v)$. This in particular implies that \eqref{eq:gleichheit} must hold with equality. By the uniqueness in 
the minimisation problem for $w$ one concludes that $ w= u \vert_{B_r(x_0)}$, implying biharmonicity of $u$ on $B_r(x_0)$. Since $x_0 \in \{ u = 0 , \nabla u \neq 0 \}$ was arbitrary we infer biharmonicity in (an open neighbourhood of) $\{ u = 0, \nabla u \neq 0 \}$.
 \end{proof}

\subsection{A global $BMO$-estimate for the Laplacian of Navier minimisers in balls under constant boundary conditions}\label{sec:app_BMO}

In this section we prove a  $BMO$-estimate \emph{up to the boundary} for the Laplacian of minimisers
in $\mathscr{N}$ in the (unit) ball $\Omega=B_1(0)\subset \mathbb{R}^n$, with $n\in\mathbb{N}$ arbitrary, in the special case of constant Navier boundary conditions $\varphi=u_0\equiv \operatorname{const}>0$. As is shown in Corollary~\ref{cor:global_Sobolev_regularity} to Theorem~\ref{thm:app_Delta_BMO} below this suffices to see that $u\in W^{2,q} (B_1(0))$ for every $q\in [1,\infty)$.

In the proof of Theorem~\ref{thm:app_Delta_BMO} below we need to find a weak solution of the biharmonic equation
under mixed Dirichlet-Navier boundary conditions. The key issue for this is the following definition of a suitable function space.
\begin{lemma}\label{lem:mixed_Dirichlet:Navier}
	Let $B_1, B_2 \subset \mathbb{R}^n$ be two balls with $G:= B_1\cap B_2\not=\emptyset$. Then
	$$
	{\mathscr H}:= \{v\in H^1_0\cap H^2 (B_2):v(x)\equiv 0 \mbox{\ a.e. on\ } B_2\setminus G\}
	$$
	together with the usual $H^2$-norm or, equivalently,
	$$
	\| v\|_{\mathscr H}^2 =\int_{B_2}  (\Delta v)^2\, \mathrm{d}x
	=\int_{G}  (\Delta v)^2\, \mathrm{d}x
	$$
	and the corresponding scalar products is a closed subspace of the Hilbert space $H^1_0\cap H^2 (B_2)$ and hence a Hilbert space itself. 
	
	Let $u\in H^2(B_2)$ be arbitrary. Minimising $v\mapsto \int_{B_2} (\Delta v)^2\, \mathrm{d}x$ on the affine space $u+{\mathscr H}$ yields a weak solution $\tilde h$ of the mixed  Dirichlet-Navier boundary value problem
	\begin{equation}\label{eq:weak_mixed_bvp}
	\begin{cases}
	\Delta^2 \tilde h =0 & \mbox{\ in\ }G,\\
	\tilde h=u,\quad \Delta \tilde h=0  & \mbox{\ on\ }\partial G\cap B_1,\\
	\tilde h=u,\quad \nabla \tilde h=\nabla u  & \mbox{\ on\ }\partial G\cap B_2.
	\end{cases}
	\end{equation}
\end{lemma}

\begin{proof}
	The Riesz-Fischer theorem yields the closedness of ${\mathscr H}$. That 
	$v\mapsto \int_{B_2} (\Delta v)^2\, \mathrm{d}x$ has a minimum $\tilde h$ on the affine space $u+{\mathscr H}$ follows by adapting Dirichlet's classical  principle. As a necessary condition we obtain the following Euler-Lagrange-equation:
	$$
	\tilde h-u\in {\mathscr H},\qquad \forall \psi\in {\mathscr H}:
	\quad \int_G \Delta \tilde h \cdot \Delta \psi\, \mathrm{d}x=0.
	$$
	This is the weak (variational) formulation of \eqref{eq:weak_mixed_bvp}.
	We emphasise that the ``weak'' attainment of the Navier boundary data is encoded in the space ${\mathscr H}$ of admissible testing functions. 
\end{proof}

\begin{theorem}\label{thm:app_Delta_BMO}
	Let $\Omega=B_1(0)\subset\mathbb{R}^n$ be 
	the (unit) ball
	and  $\varphi=u_0\equiv \operatorname{const}>0$. We consider any minimiser $u$ of $F$ on  ${\mathscr N}$. 
	We consider
	$$
	U\in L^2_{\operatorname{loc}}(\mathbb{R}^n),\quad
	U(x)=\begin{cases}
	\Delta u (x)  & \mbox{\ in\ }\Omega,\\
	-|x|^{2-n}\Delta u(x/|x|^2)  & \mbox{\ in\ }\mathbb{R}^n\setminus \Omega.
	\end{cases}
	$$
	Then for any $R_0\in (0,\frac{1}{8} )$ there exists a constant $C=C(u,R_0)$ such that for all $x_0\in\overline\Omega $ 
	 and all $r\in(0,R_0)$ one has the following estimate of bounded mean oscillation type
	\begin{equation}\label{eq:app_2.1}
	\int_{B_r (x_0)}|U -U_{x_0,r}|^2\, \mathrm{d}x\le Cr^n,
	\end{equation}
	where
	$$
	U_{x_0,r}:=\fint_{B_r(x_0)} U(x)\, \mathrm{d}x:= \frac1{|B_r(x_0)|}\int_{B_r(x_0)} U(x)\, \mathrm{d}x
	$$
	denotes the mean value of $U$ in $B_r(x_0)$. 
\end{theorem}
We have to leave the question open whether \eqref{eq:app_2.1} holds for every $x_0\in \mathbb{R}^n$. For our purposes, however, the previous result is strong enough. This will be shown in Corollary~\ref{cor:global_Sobolev_regularity} below.

For the existence of such minimisers we refer again to Theorem~\ref{thm:existence}.

\begin{proof} We choose some $R_0\in (0,\frac18 )$ and keep it fixed in what follows.
	Let $0<r<R\le R_0$ and $x_0\in \overline\Omega$ be arbitrary. 
If $|x_0|\leq  \frac{1}{2}$ then $\mathrm{dist}(x_0,\partial\Omega) \geq \frac{1}{2} > \frac{3}{8} > 3R_0$ and hence the desired estimate follows from Theorem~\ref{thm:Delta_BMO}. Hence we may assume that $|x_0| > \frac{1}{2}$. In particular one has also $0 \not \in B_{2R_0}(x_0)$.

	We keep $x_0$ and $R_0$ fixed in what follows. In the following argument we mainly have the case in mind that $B_{2R}(x_0) \cap \partial \Omega \neq \emptyset$. If on contrary $B_{2R}(x_0) \subset \Omega$, the reasoning below still applies, but the modifications are void and it runs along the lines of  \cite[Theorem 1.1]{DiKaVa}.

For regularity purposes we first take a function  
	$$
	\tilde u\in u_0 +( H^1_0\cap H^2) (\Omega ),\quad 
	\tilde u(x)=
	\begin{cases}
	u(x)\quad & \mbox{\ for\ } x\in \overline\Omega \setminus B_{7R/4} (x_0),\\
	u_0 \quad & \mbox{\ for\ } x\in \overline\Omega \cap B_{5R/4} (x_0).
	\end{cases}
	$$
 We apply Lemma~\ref{lem:mixed_Dirichlet:Navier} with $B_1=B_{2R}(x_0)$ and 
	$B_2=\Omega =B_1(0)$. We recall the Hilbert space 
	$$
	{\mathscr H}=\{v\in H^1_0\cap H^2 (\Omega):v(x)\equiv 0 \mbox{\ a.e. on\ } \Omega \setminus B_{2R}(x_0)\}
	$$ 
	from there and find a minimum $\tilde h \in \tilde u+{\mathscr H}$
	of $v\mapsto \int_{B_{2R}(x_0)\cap\Omega}(\Delta v)^2\, \mathrm{d}x$ on $\tilde u+{\mathscr H}$. This minimiser weakly solves
	$$
	\begin{cases}
	\Delta^2 \tilde h =0 & \mbox{\ in\ }\Omega\cap B_{2R}(x_0),\\
	\tilde h=u_0,\quad \Delta \tilde h=0 & \mbox{\ on\ }\partial \Omega\cap B_{2R}(x_0),\\
	\tilde h=u,\quad \nabla \tilde h=\nabla u  & \mbox{\ on\ }\partial B_{2R}(x_0)\cap \Omega.
	\end{cases}
	$$
	By elliptic regularity $\tilde h\in C^\infty(\Omega \cap B_{2R}(x_0))$
	and moreover $\tilde h\in C^\infty(\overline{\Omega \cap B_{R}(x_0)})$. This means that we have 
	$$
	\tilde h=u_0,\quad \Delta \tilde h=0  \mbox{\ on\ }\partial \Omega\cap \overline{B_{R}(x_0)}
	$$
	in the sense of smooth functions.
	We define $h \in \mathscr{N}$ via
	$$
	h(x):=
	\begin{cases}
	\tilde h(x) \quad &\quad \mbox{for}\quad x\in\Omega \cap  B_{2R}(x_0),\\
	u(x) \quad &\quad \mbox{for}\quad x\in \Omega \setminus B_{2R}(x_0).
	\end{cases}
	$$
	By minimality of $u$ we have that $F(u)\le F(h)$ which implies that
	$$
	\int_{B_{2R}(x_0)\cap\Omega}(\Delta u)^2\, \mathrm{d}x +|\{x\in B_{2R}(x_0)\cap \Omega :u(x)\not=0 \}|\le 
	\int_{B_{2R}(x_0)\cap \Omega }(\Delta h)^2\, \mathrm{d}x +|\{x\in B_{2R}(x_0)\cap \Omega :h(x)\not=0 \}|.
	$$
	This yields
	\begin{equation}\label{eq:app_2.2}
	\int_{B_{2R}(x_0)\cap\Omega}\left( (\Delta u)^2- (\Delta h)^2\right)\, \mathrm{d}x\le C R^n
	\end{equation}
	with a universal constant $C>0$. Since 
	$h|_{B_{2R}(x_0)\cap\Omega}=\tilde h \in \tilde u+{\mathscr H}$
	minimises 
	$v\mapsto \int_{B_{2R}(x_0)\cap\Omega}(\Delta v)^2\, \mathrm{d}x$ 
	on $\tilde u+{\mathscr H}$ and 
	$u-h\in \mathscr{H}$ is an admissible testing function we have
	$$
	\int_{B_{2R}(x_0)\cap \Omega}\left( \Delta u-\Delta h\right)\,\Delta \tilde h \, \mathrm{d}x
	=\int_{B_{2R}(x_0)\cap \Omega}\left( \Delta u-\Delta h\right)\,\Delta h \, \mathrm{d}x=0.
	$$ 
	This gives
	\begin{align*}
	\int_{B_{2R}(x_0)\cap \Omega} & \left( (\Delta u)^2- (\Delta h)^2\right)\, \mathrm{d}x
	= \int_{B_{2R}(x_0)\cap \Omega}\left( \Delta u-\Delta h\right)\,\left( \Delta u+\Delta h\right)\, \mathrm{d}x\\
	=& \int_{B_{2R}(x_0)\cap \Omega}\left( \Delta u-\Delta h\right)\,\left( \Delta u-\Delta h\right) \, \mathrm{d}x
	=  \int_{B_{2R}(x_0)\cap \Omega}\left( \Delta u-\Delta h\right)^2 \, \mathrm{d}x.
	\end{align*}
	With this, we conclude from \eqref{eq:app_2.2} that 
	\begin{equation}\label{eq:app_2.3_1}
	\int_{B_{2R}(x_0)\cap \Omega} \left( \Delta u-\Delta h\right)^2 \, \mathrm{d}x\le C R^n.
	\end{equation}
	Since $\Delta u $ and $\Delta h $ 
	vanish on ${\partial \Omega}$ at least in a variational sense, it is a natural idea to introduce the following odd ``Kelvin transformed'' extensions
	$U,H\in L^2_{\operatorname{loc}}(\mathbb{R}^n)$
	\begin{equation}\label{eq:odd_extension}
	U(x):=
	\begin{cases}
	\Delta u(x)\quad &\mbox{\ if\ }|x|\le 1,\\
	-|x|^{2-n}\Delta u(x/|x|^2)\  &\mbox{\ if\ }|x| > 1,
	\end{cases}
	\quad 
	H(x):=
	\begin{cases}
	\Delta h(x)\quad &\mbox{\ if\ }|x|\le 1,\\
	-|x|^{2-n}\Delta h(x/|x|^2)\  &\mbox{\ if\ }|x| > 1.
	\end{cases}
	\end{equation}
	One should observe that (thanks to using $\tilde u$ instead of $u$ for introducing $\tilde h$) we have in a classical sense
	$$
	\Delta H=0\quad \mbox{in}\quad  \overline\Omega \cap B_R(x_0) ,
	\qquad H=0\quad  \mbox{on}\quad  \partial \Omega \cap B_R(x_0).
	$$
	This yields that we have also for the extended function
	\begin{equation}
	\Delta H=0\quad \mbox{in} \;  B_R(x_0).
	\end{equation}
	For this one should have in mind:
	\begin{itemize}
		\item $B_R(x_0)$ is a subset of the union of $\overline\Omega \cap B_R(x_0)$ and its inversion. In order to see this we take $x\in B_R(x_0)$ with $|x|>1$ and need to show that $x/|x|^2\in B_R(x_0)$. This follows in turn from the following inequalities:
		\begin{align*}
		\left| \frac{x}{|x|^2}-x_0\right|^2 =& \frac{1}{|x|^2} \left( 1-2x\cdot x_0+|x_0|^2\, |x|^2 \right)
		= \frac{1}{|x|^2} \left( \underbrace{|x-x_0|^2}_{<R^2}+\underbrace{(1-|x_0|^2)}_{\ge 0}\,\underbrace{(1- |x|^2)}_{<0} \right)\\
		<&\frac{R^2}{|x|^2}<R^2.
		\end{align*}
		\item $H\in C^\infty (B_R(x_0) \setminus \partial \Omega )\cap C^1(B_R(x_0))$
		and harmonic in $B_R(x_0) \setminus \partial \Omega$. Hence, it is weakly and consequently classically harmonic in $B_R(x_0)$.
	\end{itemize}
	Since the extension operator, defined in \eqref{eq:odd_extension}
	$$
	L^2(\Omega\setminus B_{1/4} (0))\to L^2 (B_4(0)\setminus B_{1/4} (0))
	$$
	is bounded and $B_{2R}(x_0) \cap B_{\frac{1}{4}}(0) = \emptyset$ (as $|x_0|> \frac{1}{2}$ and $R_0< \frac{1}{8}$), we conclude from  \eqref{eq:app_2.3_1} that 
	\begin{equation}\label{eq:app_2.3}
	\int_{B_{2R}(x_0)} \left( U- H\right)^2 \, \mathrm{d}x\le
	C\, \int_{B_{2R}(x_0) \cap \Omega} \left( U- H\right)^2 \, \mathrm{d}x\le C R^n.
	\end{equation}
	By means of H\"older's inequality
	$$
	|U_{x_0,r} - H_{x_0,r} |^2\le \left(\fint_{B_r(x_0)} |U -H|\, \mathrm{d}x\right)^2 \le \fint_{B_r(x_0)} |U -H|^2\, \mathrm{d}x,
	$$
	we conclude from \eqref{eq:app_2.3} that 
	\begin{equation}\label{eq:app_2.4}
	\int_{B_r(x_0)}|U_{x_0,r} - H_{x_0,r} |^2\, \mathrm{d}x\le \int_{B_r(x_0)} |U - H|^2\, \mathrm{d}x
	\le C R^n.
	\end{equation}
	We deduce  now 
	a Campanato type inequality for $H$, 
	which is harmonic in $B_{R}(x_0)$.  
	
	We fix an arbitrary $\alpha\in (0,1)$.
	For any harmonic function $f$ on $B_{R}(x_0)$ local elliptic estimates yield {(with $[f]_{\beta,G}$ denoting the Hölder seminorm in $C^{0,\beta}(G)$)}
	$$
	[f]_{\alpha/2,B_{R/4}(x_0) }\le C\, R^{-(\alpha+n)/2}\| f\|_{L^2(B_{R/2}(x_0))}
	\quad\mbox{with a constant}\quad C=C(n,\alpha).
	$$
	If we assume further that $f(x_0)=0$ we conclude that $\forall r\in (0,\frac{R}{4})$:
	$$
	\int_{B_r(x_0)} |f(x)|^2\, \mathrm{d}x\le[f]_{\alpha/2,B_{R/4}(x_0) }^2\int_{B_r(x_0)} |x-x_0|^\alpha \, \mathrm{d}x\le C\left(\frac{r}{R} \right)^{n+\alpha}\int_{B_R(x_0)} |f(x)|^2\, \mathrm{d}x.
	$$ 
	For $r\in (\frac{R}{4},R)$, $\frac{r}{R}$ is bounded from below and we directly 
	see that
	$$
	\int_{B_r(x_0)} |f(x)|^2\, \mathrm{d}x\le\int_{B_R(x_0)} |f(x)|^2\, \mathrm{d}x\le C\left(\frac{r}{R} \right)^{n+\alpha}\int_{B_R(x_0)} |f(x)|^2\, \mathrm{d}x.
	$$ 
	All in all we have 
	$$
	\forall r\in (0,R):\quad 
	\int_{B_r(x_0)} |f(x)|^2\, \mathrm{d}x\le  C\left(\frac{r}{R} \right)^{n+\alpha}\int_{B_R(x_0)} |f(x)|^2\, \mathrm{d}x.
	$$
	with a constant $C=C(n,\alpha)$.
	We apply this to the harmonic function (making use of its mean value property in $B_R(x_0)$)
	$$
	f(x)=H(x)-H(x_0)=H(x)-H_{x_0,r}
	=H(x)-H_{x_0,R}
	$$
	and find with a constant $C=C(n,\alpha)$  that
	\begin{equation}\label{eq:app_2.5}
	\int_{B_r(x_0)}|H - H_{x_0,r} |^2\, \mathrm{d}x\le 
	C\left(\frac{r}{R} \right)^{n+\alpha} \int_{B_R(x_0)}|H - H_{x_0,R} |^2\, \mathrm{d}x.
	\end{equation}
	Putting all the estimates together we find the following with constants $C=C(n,\alpha)$:
	\begin{align*}
	\int_{B_r (x_0)}  & |U -U_{x_0,r}|^2\, \mathrm{d}x
	= \int_{B_r (x_0)}|U -H + H - H_{x_0,r} + H_{x_0,r}- U_{x_0,r}|^2\, \mathrm{d}x\\
	\le& C\left( \int_{B_r (x_0)}|U -H|^2 \, \mathrm{d}x + \int_{B_r (x_0)} |H - H_{x_0,r}|^2 \, \mathrm{d}x + \int_{B_r (x_0)} |H_{x_0,r}- U_{x_0,r}|^2\, \mathrm{d}x \right) \\
	\stackrel{\mbox{\scriptsize \eqref{eq:app_2.3},\eqref{eq:app_2.4},\eqref{eq:app_2.5}}}{\le}&
	C\left( R^n +  \left(\frac{r}{R} \right)^{n+\alpha} \int_{B_R(x_0)}|H - H_{x_0,R} |^2\, \mathrm{d}x  \right)\\
	\le &C\Bigg( R^n + \left(\frac{r}{R} \right)^{n+\alpha} \\
	&\cdot \left(  \int_{B_R (x_0)}|H - U|^2 \, \mathrm{d}x + \int_{B_R (x_0)} | U - U_{x_0,R}|^2 \, \mathrm{d}x
	+ \int_{B_R (x_0)} | U_{x_0,R}- H_{x_0,R}|^2\, \mathrm{d}x\right) \Bigg)\\
	\stackrel{\mbox{\scriptsize \eqref{eq:app_2.3},\eqref{eq:app_2.4}}}{\le}
	&C\left( R^n + \left(\frac{r}{R} \right)^{n+\alpha} 
	\cdot \left(  R^n + \int_{B_R (x_0)} | U - U_{x_0,R}|^2 \, \mathrm{d}x \right) \right)\\
	\le& C\left( R^n + \left(\frac{r}{R} \right)^{n+\alpha} 
	\cdot  \int_{B_R (x_0)} | U - U_{x_0,R}|^2 \, \mathrm{d}x  \right).
	\end{align*}
	Introducing the notation
	$$
	\Phi(r):= \int_{B_r(x_0)}| U - U_{x_0,r} |^2\, \mathrm{d}x
	$$
	this estimate rewrites as
	\begin{equation*}
	\forall 0 <r\le R\le R_0:\quad \Phi(r)\le C \left(  R^n + \left(\frac{r}{R} \right)^{n+\alpha}\Phi(R)\right) .
	\end{equation*}
	In order to proceed we need to introduce the following increasing variant of $\Phi$:
	$$
	\widetilde{\Phi} (r) := \sup_{\rho\in (0,r]}\Phi(\rho),
	$$
	which obeys the same relation as $\Phi$:
	\begin{equation}\label{eq:app_2.6}
	\forall 0 <r\le R\le R_0:\quad \widetilde{\Phi}(r)\le C \left(  R^n + \left(\frac{r}{R} \right)^{n+\alpha}\widetilde{\Phi}(R)\right) .
	\end{equation}
	Lemma 2.1 from \cite[Chapter III]{Giaquinta} yields that then with $C=C(n,\alpha)$
	$$
	\forall 0 <r\le R\le R_0:\quad \widetilde{\Phi}(r)\le C \left(  r^n + \left(\frac{r}{R} \right)^{n}\widetilde{\Phi}(R)\right) 
	$$
	and in particular that
	\begin{align*}
	\forall 0 <r\le R_0:\quad \Phi(r)\le &\widetilde{\Phi} (r)
	\le  C r^n \left(  1 + \frac{1}{R_0^n} \widetilde{\Phi}(R_0)\right) 
	\le C r^n \left(  1 + \frac{1}{R_0^n} F(u)\right)\\
	\le& C(n,\alpha, F(u), R_0)r^n.
	\end{align*}
	(Recall again the local boundedness of the extension operator in \eqref{eq:odd_extension}.)
	Since we may e.g. fix $\alpha=1/2$ we end up with
	$$
	\forall 0 <r\le R_0:\quad \int_{B_r(x_0)}|U - U_{x_0,r} |^2\, \mathrm{d}x\le C r^n
	$$
	with a constant $C=C(n,R_0,F(u))$.
	This proves \eqref{eq:app_2.1}. 
\end{proof}

\begin{corollary}\label{cor:global_Sobolev_regularity}
	Let $\Omega=B_1(0)\subset\mathbb{R}^n$ be 
	the (unit) ball
	and  $\varphi=u_0\equiv \operatorname{const}>0$. We consider any minimiser $u$ of $F$ on  ${\mathscr N}$. 	Then we have that $u \in W^{2,q} (B_1(0))$ for all $q \in [1,\infty)$ and that  $u \in C^{1,\alpha}(\overline{B_1(0)})$ for each $\alpha \in (0,1)$.
\end{corollary}

\begin{proof}
	According to \cite[Chapter IV, Remark 1.1.1]{Stein} the definition of $BMO$ does not depend on whether one is working with balls or with cubes. Hence the previous Theorem~\ref{thm:app_Delta_BMO} shows that $\Delta u \in BMO(B_1(0))$ in the sense of Jones \cite[p. 41]{Jones}. According to the remarks between 
	Theorem 2 and Theorem 3 in Jones' work, his Theorem 1 (where for the sufficiency part he gives strong credits to Reimann~\cite{Reimann}) applies in particular to the ball $B_1(0)$. This yields that $\Delta u $ has an extension $\tilde U\in BMO (\mathbb{R}^n)$. Applying  \cite[Corollary in Chapter IV.1.3]{Stein} shows then that $\tilde U\in L^q_{\operatorname{loc}} (\mathbb{R}^n)$ for any $q\in (1,\infty)$. Hence $\Delta u \in L^q(B_1(0))$
	and elliptic regularity yields that $u \in W^{2,q} (B_1(0))$.
	By Sobolev embedding the H\"older regularity follows.
\end{proof}

\begin{remark}\label{app_rem:2}
 	We know from Remark~\ref{rem:1.2} that in dimensions $n\in\{1,2,3\}$ the previous boundary regularity result can be extended to general sufficiently smooth strictly positive 
 	boundary data, to general sufficiently smooth domains, and to the case of Dirichlet boundary conditions.
	
	We have to leave the 
 question open whether such generalisations are available in dimensions $n\ge 4$.
\end{remark}
 

\section{{Radial minimisers}} \label{sec:radial}

In Remark \ref{rem:radial} we have seen that minimisers in $\mathscr{D}_{\operatorname{rad}}$ and $\mathscr{N}_{\operatorname{rad}}$  can be found.  In the sequel we want to compute these minimisers explicitly and study their properties on balls $\Omega=B_R(0)$ with $\varphi\equiv \operatorname{const} =: u_0\in (0,\infty)$. 

Recall that each radial biharmonic function is given by 
\begin{equation*}
f(r) = C_1 r^{4-n} + C_2 r^{2-n} + C_3 r^2 + C_4
\end{equation*}
if $n \neq 2,n\neq 4$,
\begin{equation*}
f(r) = C_1 \log r + C_2 r^{-2}  + C_3 r^2 + C_4
\end{equation*}
if $n = 4$ and 
\begin{equation*}
f(r) = C_1 r^2 \log r + C_2 \log r + C_3 r^2 + C_4
\end{equation*}
if $n = 2$. Hence each radial minimiser must be either of the form $u \equiv u_0$ or 
\begin{equation}\label{eq:sonne}
u(x) = u(\rho,C_1,C_2,C_3,C_4)(x) = \begin{cases} 0 & |x| \leq \rho, \\ f(|x|) & |x| \geq \rho.  \end{cases}
\end{equation}
 We remark that $C_1,C_2,C_3,C_4$ can be uniquely determined by $\rho$ via the boundary conditions. 

\begin{example}\label{eq:Naviern=2}
 Here we look at Navier boundary values in dimension $n = 2$ on $\Omega =B_1(0)$ with $\varphi(x)\equiv u_0\in (0,\infty)$.
 We also make a slight modification of the functional, namely we seek to minimise 
\begin{equation}
F_\lambda(u) = \int_\Omega (\Delta u)^2 \; \mathrm{d}x + \lambda |\{ u \neq 0 \}|,
\end{equation}
for some parameter $\lambda > 0$.  One either has $u \equiv u_0$, i.e. $F(u) = \lambda \pi$ or $u = u(\rho,C_1,C_2,C_3,C_4)$ for some $\rho \in (0,1)$ and $C_1,C_2,C_3,C_4\in \mathbb{R}$, cf. \eqref{eq:sonne}. 
If $u = u(\rho, C_1,C_2,C_3,C_4)$ is as above then 
\begin{equation}\label{eq:19}
\Delta u(x) = C_1 ( 4 \log r + 4 ) + 4C_3. 
\end{equation}
The Navier boundary conditions imply then $4C_1 + 4C_3 = 0$ and $C_3 + C_4 = u_0$. In particular $C_1= -C_3 = C_4-u_0$. The conditions $u = \partial_r u = 0$ on $\partial B_\rho(0)$ yield 
\begin{equation}\label{eq:21}
 0  = C_1 \rho^2 \log \rho  + C_2 \log \rho + C_3 \rho^2 + C_4
\end{equation}
and 
\begin{equation*}
0 =  C_1 \left( \rho + 2\rho \log \rho \right) + C_2 \frac{1}{\rho} + 2 C_3 \rho. 
\end{equation*}
Using $C_3 = - C_1$ the last equation simplifies to 
\begin{equation*}
0  = C_1 (- \rho + 2 \rho \log \rho ) + C_2 \frac{1}{\rho},
\end{equation*}
i.e. 
\begin{equation*}
C_2 = (\rho^2 - 2 \rho^2 \log \rho) C_1. 
\end{equation*}
Now \eqref{eq:21} yields 
\begin{equation*}
0 = C_1 \rho^2 \log \rho + \log \rho (\rho^2 - 2\rho^2 \log \rho) C_1 - C_1 \rho^2 + (C_1+ u_0). 
\end{equation*}
Therefore 
\begin{equation*}
C_1 = - \frac{u_0}{2\rho^2\log \rho - 2\rho^2 \log^2 \rho - \rho^2 + 1 }.
\end{equation*}
Going back to \eqref{eq:19} we find $\Delta u(x) =  4C_1 \log r$ and hence 
\begin{align*}
F_\lambda(u) =& \lambda \pi ( 1- \rho^2 ) 
+  32\pi C_1^2\int_{\rho}^1 r \log^2r {\; \mathrm{d}r} = \lambda\pi ( 1- \rho^2 ) + 8\pi  C_1^2 ( 1- \rho^2 -2\rho^2 \log^2 \rho + 2 \rho^2 \log \rho)\\
=&\lambda \pi ( 1- \rho^2 ) + \frac{8\pi  u_0^2 ( 1- \rho^2 -2\rho^2 \log^2 \rho + 2 \rho^2 \log \rho)}{(2\rho^2\log \rho - 2\rho^2 \log^2 \rho - \rho^2 + 1)^2},  
\end{align*}
i.e. 
\begin{equation*}
F_\lambda(u) =  f_\lambda(u_0,\rho),
\end{equation*}
where
$$
f_\lambda(u_0,\rho):=  \lambda \pi ( 1- \rho^2 ) + \frac{8\pi  u_0^2 }{2\rho^2\log \rho - 2\rho^2 \log^2 \rho - \rho^2 + 1}.
$$
In particular we find 
\begin{equation}\label{eq:infvalueNrad}
\inf_{v \in \mathscr{N}_{\operatorname{rad}}} F_\lambda(v) = \min \left\lbrace
 \lambda \pi , \inf_{\rho \in (0,1)} f_\lambda(u_0,\rho)\right\rbrace.
\end{equation}

One observes that $f_\lambda(u_0,0)=\lambda\pi+8\pi\, u_0^2$ and $\lim_{\rho\uparrow 1}f_\lambda(u_0,\rho)=\infty$.
Moreover, $f_\lambda (u_0,\rho)=\lambda f_1(u_0/\sqrt{\lambda},\rho)$, i.e., enlarging $\lambda$ has the same effect as making $u_0$ smaller.
The infimum of $f_\lambda(u_0,\,.\,)$ can be studied  computer-assistedly. For some plots in the case of $\lambda = 1$ one may see Figure~\ref{fig:radial_minima}.
\begin{figure}[h] 
	\centering 
	\includegraphics[width=.3\textwidth]{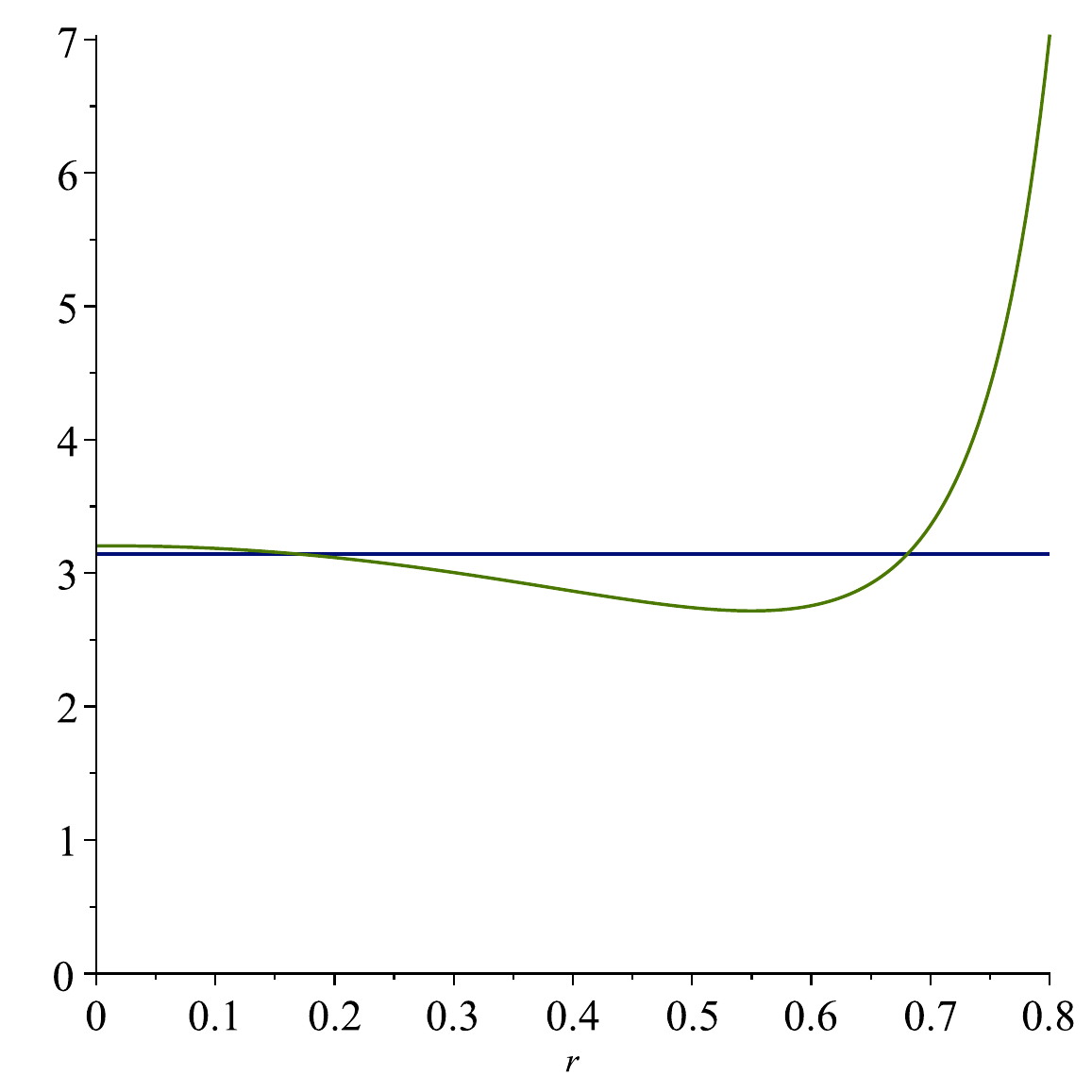}\hspace{0.3cm}
	\includegraphics[width=.3\textwidth]{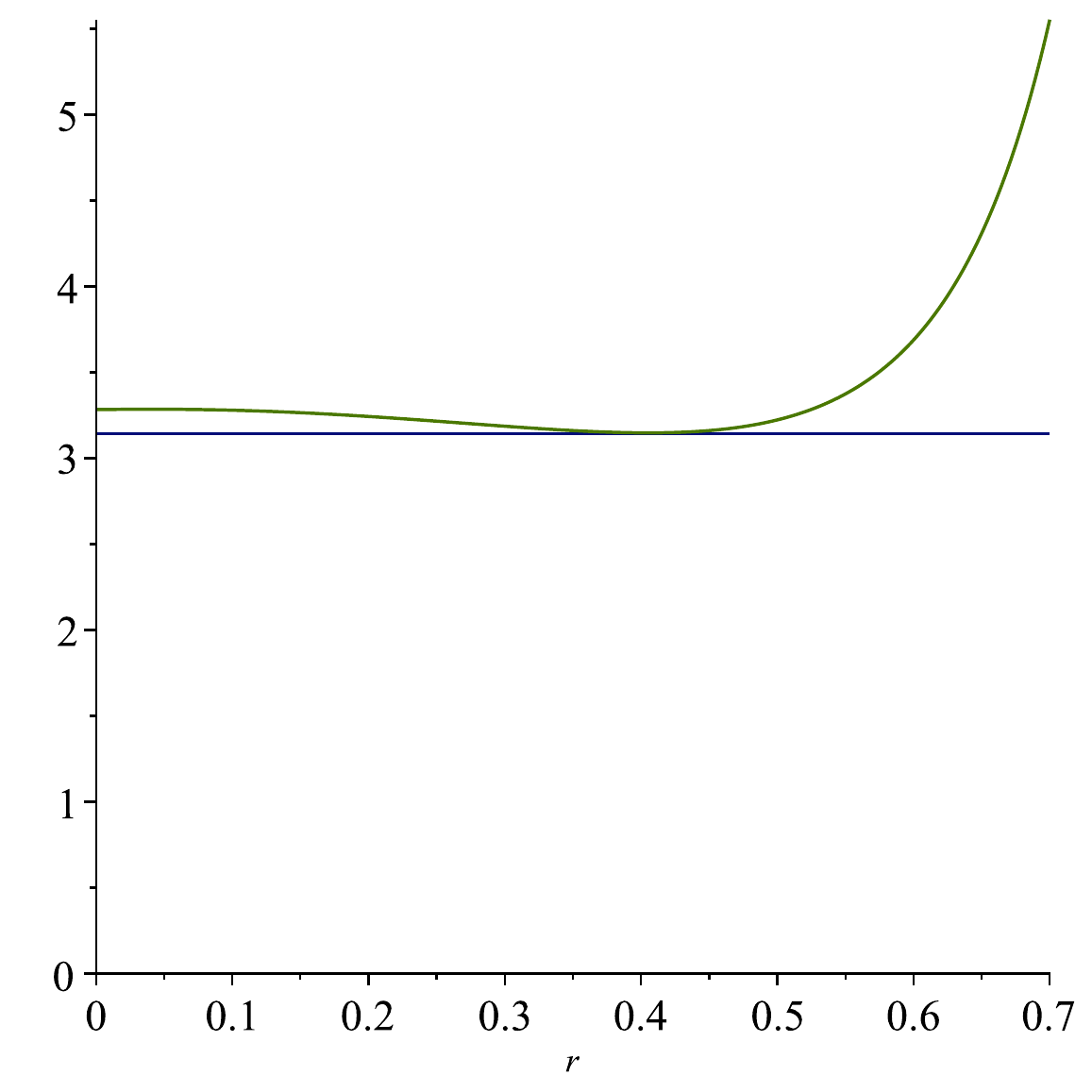}\hspace{0.3cm}
	\includegraphics[width=.3\textwidth]{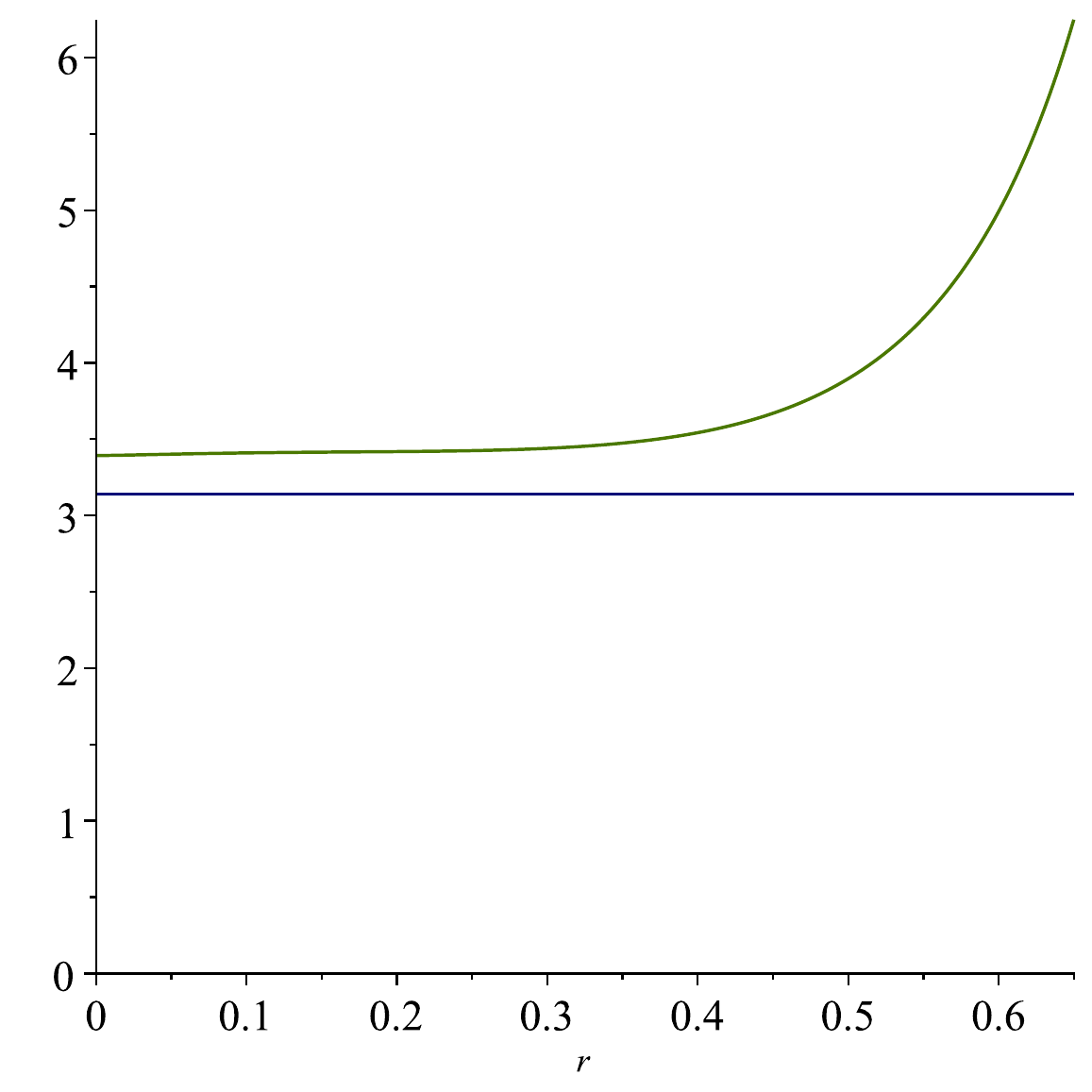} 
	\caption[] 
	{Plots of $f_1(0.05,\,.\,)$, $f_1(0.075,\,.\,)$, and $f_1(0.1,\,.\,)$
		(left to right). The straight line marks the threshold energy level $\pi$.}
	\label{fig:radial_minima} 
\end{figure}

This shows that for small enough $u_0$ there is a ``flat'' set, while for large $u_0$ there is none.

\end{example}

\begin{remark}\label{rem:4}
We have seen that for radial minimisers in $\mathscr{N}_{\operatorname{rad}}, n=2$ one has $\Delta u \not \in C^0(B_1(0))$ unless $u \equiv \operatorname{const}$. Indeed, one infers from the previous example that if the infiumum in
\eqref{eq:infvalueNrad} is smaller than $\lambda \pi$ one has a minimiser $u \in \mathscr{N}_{\operatorname{rad}}$ with 
\begin{equation}
    \Delta u (x) = \begin{cases}
         0, & |x| \leq \rho_{\min}, \\ - \frac{4u_0}{2\rho_{\min}^2\log(\rho_{\min})- 2 \rho_{\min}^2 \log^2(\rho_{\min}) - \rho_{\min}^2 + 1} \log|x|, & |x| > \rho_{\min},
    \end{cases}
\end{equation}
where $\rho_{\min} \in (0,1)$ is some value where the infimum in \eqref{eq:infvalueNrad} is attained.  
This discontinuity phenomenon also occurs in arbitrary dimension. Indeed, here we show the following 
\\
\textbf{Claim.} Suppose that $u \in \mathscr{N}_{\operatorname{rad}}$ is a minimiser with $|\{u= 0 \}| >0$. Then $\Delta u \not \in C^0(B_1(0))$. In particular $u \not \in C^2(B_1(0))$. 
\\
\textit{Proof of the claim.} Assume that $u \in \mathscr{N}_{\operatorname{rad}}$ is as in the claim and satisfies $\Delta u \in C^0(B_1(0))$. By Remark \ref{rem:1.2} we also have that $\Delta u \in C^0(\overline{B_1(0)})$ and $\Delta u \vert_{\partial B_1(0) } = 0$.
Then $\{ u = 0 \} = \overline{B_\rho(0)}$ for some $\rho \in (0,1)$. Let $A_{1,\rho} := B_1(0) \setminus \overline{B_\rho(0)}$. 
Now $\Delta u$ lies in $C^2(A_{1,\rho}) \cap C^0( \overline{A_{1,\rho}}) $ and is harmonic in $A_{1,\rho}$. The classical maximum principle now yields 
\begin{equation}\label{eq:maxpr}
\|\Delta u \|_{L^\infty(A_{1,\rho})} \leq \|\Delta u \|_{L^\infty( \partial A_{1,\rho} ) }. 
\end{equation}
However, note that $\Delta u= 0 $ on $\partial B_1(0)$ and $\Delta u = 0 $ on $\partial B_\rho(0)$ since $u = 0$ on $B_\rho(0)$ and $\Delta u \in C^0$. As a consequence $\Delta u \equiv 0$ on $\partial A_{1,\rho}$ and \eqref{eq:maxpr} yields $\Delta u = 0$ on $A_{1,\rho}$.  Since $\Delta u = 0$ also on $B_\rho(0)$ we conclude that $\Delta u = 0$ on the whole of $B_1(0)$. This  {yields} that $u$ is harmonic on $B_1(0)$ and thereupon the maximum principle implies that $u(x)\equiv u_0 > 0$ (as $u\vert_{\partial \Omega} = u_0 > 0)$. This contradicts $|\{ u = 0 \}| > 0$ and the claim follows. 
\end{remark}

\begin{remark}\label{rem:strictlyincr}
The expression in \eqref{eq:infvalueNrad} is strictly increasing in $u_0$ until it reaches the value $\lambda \pi$. {From} there on it is constant. Also this is true in any arbitrary dimension, as we shall prove here in the case of $\lambda =1$.

\medskip\noindent
\textbf{Claim.}  
If (for $B_1(0) \subset \mathbb{R}^n$) we set
\begin{equation*}
    \mathscr{N}_{\operatorname{rad}}(u_0) := \{ u \in H^2(B_1(0)) : u -u_0 \in H_0^1(B_1(0)) \},  
\end{equation*}
then the function $f: (0, \infty) \rightarrow (0, \infty)$
\begin{equation*}
    f(u_0) := \inf_{\psi \in \mathscr{N}_{\operatorname{rad}}(u_0)} F(\psi) 
\end{equation*}
is of the form $f(z) = \min\{ \eta(z) , |B_1(0)| \}$, where $\eta$ is a strictly increasing function. More precisely, if $\hat{u}_0$ permits a nonconstant radial minimiser then $f$ is strictly increasing on $[0,\hat{u}_0]$.

\medskip\noindent
{\emph{Proof of the claim.}}  {Let} $0 <u_1 < u_0$ be arbitrary constant boundary data. Then there exists $\alpha \in (0,1)$ such that $u_1 = \alpha u_0$. Let $u \in \mathscr{N}_{\operatorname{rad}}(u_0)$ be a minimiser. The rescaled function $u_\alpha := \alpha u$ lies in $\mathscr{N}_{\operatorname{rad}}(u_1)$ and satisfies 
\begin{align}
    \inf_{\psi \in \mathscr{N}_{\operatorname{rad}}(u_1)} F(\psi) & \leq F(\alpha u)  = \alpha^2 \int_\Omega ( \Delta u ) ^2 \; \mathrm{d}x + |\{ \alpha u \neq 0 \} | = \alpha^2 \int_\Omega ( \Delta u ) ^2 \; \mathrm{d}x + |\{  u \neq 0 \} |\nonumber
    \\ 
    & \leq \int_\Omega ( \Delta u ) ^2 \; \mathrm{d}x + |\{  u \neq 0 \} | = F(u) = \inf_{\psi \in \mathscr{N}_{\operatorname{rad}}(u_0)} F(\psi).\label{eq:29}
\end{align} 
The inequality in the first estimate \eqref{eq:29} is strict iff $\int_\Omega (\Delta u)^2 \; \mathrm{d}x \neq 0$. Notice that this integal vanishes iff $u$ is harmonic in $\Omega$. Uniqueness of solutions for the harmonic Dirichlet problem yields that in this case $u \equiv u_0 \equiv \mathrm{const}$. Hence
\begin{equation*}
    \inf_{\psi \in \mathscr{N}_{\operatorname{rad}}(u_0)} F(\psi) = F(u) = F(u_0) = |B_1(0)|.
\end{equation*}
This and \eqref{eq:29} imply that 
\begin{equation*}
   f(u_1) \leq f(u_0) \quad \textrm{where the inequality is strict, when $u_0$ admits a nonconstant minimiser.} 
\end{equation*}
The assertion follows.
\end{remark}

\begin{example}\label{ex:ridofR} 
 We  minimise $F$ in $\mathscr{N}_{\operatorname{rad}}$ again in dimension $n = 2$, but this time with $\Omega = B_R(0)$ for any arbitrary $R \in (0,\infty)$ and constant boundary value $u_0$. We write $F(\cdot \; |  \; B_R(0))$ to distinguish between the minimisation problems for different values of $R$. One readily checks that for each $u \in \mathscr{N}_{\operatorname{rad}}$ one can define $u_R :B_1(0) \rightarrow \mathbb{R}$ via $u_R(x) := u(Rx)$ and obtains
 \begin{equation*}
 F(u \; | \; B_R(0)) = R^{n-4} \left( \int_{B_1(0)} ( \Delta u_R)^2 \; \mathrm{d}x + R^4 |\{ x \in B_1(0) : u_R \neq 0\}| \right) 
 \end{equation*}
 This yields that for $\lambda =R^4$ one has 
 \begin{equation*}
 F(u \; | \; B_R(0)) = R^{n-4} F_\lambda (u_R \; | \; B_1(0)), 
 \end{equation*}
 where $F_\lambda$ is as in Example \ref{eq:Naviern=2}.
 Using this and the fact that $u\vert_{{\partial B_R(0)}} = (u_R)\vert_{{\partial B_1(0)}}  = u_0$  we obtain from \eqref{eq:infvalueNrad}
 \begin{align*}
 \inf_{u \in \mathscr{N}_{\operatorname{rad}}} F(u \; | \; B_R(0) ) & = R^{n-4}  \inf_{u \in \mathscr{N}_{\operatorname{rad}}} F_{R^4}(u \; | \; B_1(0) )
 \\ & = R^{n-4}   \min \left\lbrace
 R^4 \pi , \inf_{\rho \in (0,1)} f_{R^4}(u_0,\rho) \right\rbrace.
 \end{align*}
 This can again be studied computer assistedly.
 \end{example}

\begin{example}
We now minimise $F$ defined in $\mathscr{D}_{\operatorname{rad}}$ for $n = 2$ on $\Omega = B_1(0)$. In particular we prescribe  $u - u_0 \in H^2_0(\Omega)$ for some constant function $u_0>0$. Again either $u \equiv u_0$ or $u = u(C_1,C_2,C_3,C_4,\rho)$ for suitable values of $C_1,C_2,C_3,C_4,\rho$, cf. \eqref{eq:sonne}. 
 The conditions we have to ensure are $u = u_0, \partial_r u = 0$ on $\partial B_1(0)$ and $u= \partial_r u = 0$ on $\partial B_\rho(0)$. We compute on $[\rho,1]$:
\begin{equation*}
\partial_r u(C_1,C_2,C_3,C_4, \rho)(r)= C_1(2r\log r + r) + C_2 \frac{1}{r} + 2C_3 r. 
\end{equation*}
Evaluated at $ r = 1$  and $r  =\rho$ this amounts to 
\begin{align}
0 & = C_1 + C_2+ 2C_3 \label{eq:firstr=1},\\
 0 &  = C_1(2\rho \log \rho + \rho) + C_2 \frac{1}{\rho} + 2C_3 \rho.  \nonumber
 \end{align}
 Multiplying the first equation with $\rho$ and subtracting the second one we find 
 \begin{equation*}
 0 = - C_1(2 \rho \log \rho) + C_2 \left( \rho - \frac{1}{\rho} \right), 
 \end{equation*}
 which yields 
 \begin{equation*}
 C_2 = C_1 \frac{2\rho \log \rho}{\rho- \frac{1}{\rho}} = C_1 \frac{2\rho^2 \log \rho}{\rho^2- 1}.
 \end{equation*}
 Reinserting this into \eqref{eq:firstr=1} yields
 \begin{equation*}
 0 = C_1 \left( 1 + \frac{2\rho^2 \log \rho}{\rho^2 - 1} \right) + 2C_3 \qquad \Rightarrow \; \; C_3 = -C_1 \left( \frac{1}{2}+ \frac{\rho^2 \log \rho}{\rho^2 - 1} \right). 
 \end{equation*}
 Next we use that $u \equiv u_0$ on $\partial B_1(0)$ yields $C_3 + C_4 = u_0$. Therefore 
 \begin{align*}
 0 & = u |_{{\partial B_\rho (0) }} = C_1 \rho^2 \log \rho + C_2 \log \rho + C_3 \rho^2 + C_4 
 \\ & = u_0 + C_1 \rho^2 \log \rho + C_2 \log \rho + C_3 (\rho^2 - 1) \\ &  = u_0 + C_1 \left( \rho^2 \log \rho + \log \rho \left( \frac{2\rho^2\log \rho}{\rho^2 - 1} \right) - (\rho^2 - 1)  \left(  \frac{1}{2} + \frac{\rho^2 \log \rho}{\rho^2 - 1} \right) \right).
 \end{align*}
 This yields an explicit fomula for $C_1$, namely 
 \begin{equation*}
 C_1 = C_1(u_0 ,\rho) :=  \frac{-u_0}{\frac{2\rho^2 \log^2 \rho}{\rho^2-1} - \frac{1}{2}(\rho^2-1) } = \frac{u_0}{\frac{1}{2} ( \rho^2 - 1) - \frac{2\rho^2 \log^2\rho}{\rho^2-1}}
 =u_0\frac{2(1-\rho^2)}{4\rho^2\log^2\rho -(1-\rho^2)^2}.
 \end{equation*}
 Therefore one has also 
 \begin{align}
 C_3 = C_3(u_0, \rho) =& -\left( \frac{1}{2} + \frac{\rho^2 \log \rho}{\rho^2-1}\right) C_1(u_0, \rho) =- \left( \frac{1}{2} + \frac{\rho^2 \log \rho}{\rho^2-1}\right) \frac{u_0}{\frac{1}{2}(\rho^2-1) - \frac{2\rho^2 \log^2\rho}{\rho^2-1} } \nonumber \\
 =&u_0\frac{2\rho^2\log\rho -(1-\rho^2)}{4\rho^2\log^2\rho -(1-\rho^2)^2}. \label{eq:C_drei}
 \end{align}
 Next we use that for $u = u(C_1,C_2,C_3,C_4, \rho)$ we have 
 \begin{equation*}
 \Delta u (x) = C_1 ( 4 \log r + 4) + 4C_3. 
 \end{equation*}
Using that $C_1= C_1(u_0,\rho)$ and $C_3 = C_3(u_0, \rho)$ one can therefore determine the $F(u)$ only in terms of $u_0,\rho$. 
An explicit formula for the infimum can therefore be found. 
Finally, we can compute the energy of each admissible $u = u(C_1,C_2,C_3,C_4,\rho)$ in terms of $u_0, \rho$ and may minimise  computer-assistedly. 
\begin{align*}
F(u) & = \int_{B_1(0) \setminus B_\rho(0)} [ C_1 (4 \log|x| + 4) + 4 C_3]^2 \mathrm{d}x + \pi (1- \rho^2)
\\ & = 2\pi C_1(u_0,\rho)^2 \int_\rho^1 \left( 4 \log r + 4 - 4 \left( \frac{1}{2} + \frac{\rho^2 \log \rho}{\rho^2-1} \right) \right)^2 r \; \mathrm{d}r + \pi (1- \rho^2) 
\\ & =  2 \pi C_1(u_0,\rho)^2 \int_\rho^1 \left( 4 \log r + D(\rho) \right)^2 r \; \mathrm{d}r + \pi ( 1- \rho^2) 
\\ & = 2\pi C_1(u_0,\rho)^2 \int_\rho^1 (16 r \log^2 r + 8 D(\rho) r \log r + D(\rho)^2 r)  \; \mathrm{d}r + \pi (1-\rho^2) 
 \end{align*}
 where $D( \rho) = 2 -  4 \frac{\rho^2 \log \rho}{\rho^2 -1 }$. Using 
 \begin{equation*}
 \int_{\rho}^1 r \log ^2 r \; \mathrm{d}r  = - \frac{\rho^2}{2} \log^2 \rho + \frac{\rho^2}{2}\log \rho + \frac{1}{4}(1- \rho^2) , \quad 
 \int_{\rho}^1 r \log r \; \mathrm{d}r = - \frac{\rho^2}{2} \log \rho - \frac{1-\rho^2}{4},
 \end{equation*}
 we infer 
 \begin{align*}
 F(u) & =g(u_0, \rho), 
 \end{align*}
 where  
 \begin{align*}
 g(u_0, \rho) & := 2 \pi C_1(u_0,\rho)^2 \left( -8 \rho^2 \log^2 \rho + 8\rho^2 \log \rho + 4(1-\rho^2) \right) \\   & \quad  + 2 \pi C_1(u_0, \rho)^2 \left( - 4 \rho^2 D(\rho) \log \rho - 2 D(\rho) ( 1- \rho^2)  + \frac{1}{2}D(\rho)^2 (1- \rho^2) \right) + \pi (1- \rho^2), 
 \end{align*} 
 with 
 \begin{equation*}
 C_1(u_0, \rho) =u_0\frac{2(1-\rho^2)}{4\rho^2\log^2\rho -(1-\rho^2)^2}, \quad  D(\rho) = 2 - 4 \frac{\rho^2 \log\rho}{\rho^2-1}.
 \end{equation*}
 In particular 
 \begin{equation}\label{eq:inifiDrad}
 \inf_{ u \in \mathscr{D}_{\operatorname{rad}} } F(u)  = \min \left\lbrace \pi , \inf_{\rho \in (0, 1)} g(u_0,\rho) \right\rbrace.
\end{equation}
\end{example}

\begin{figure}[h] 
	\centering 
	\includegraphics[width=.3\textwidth]{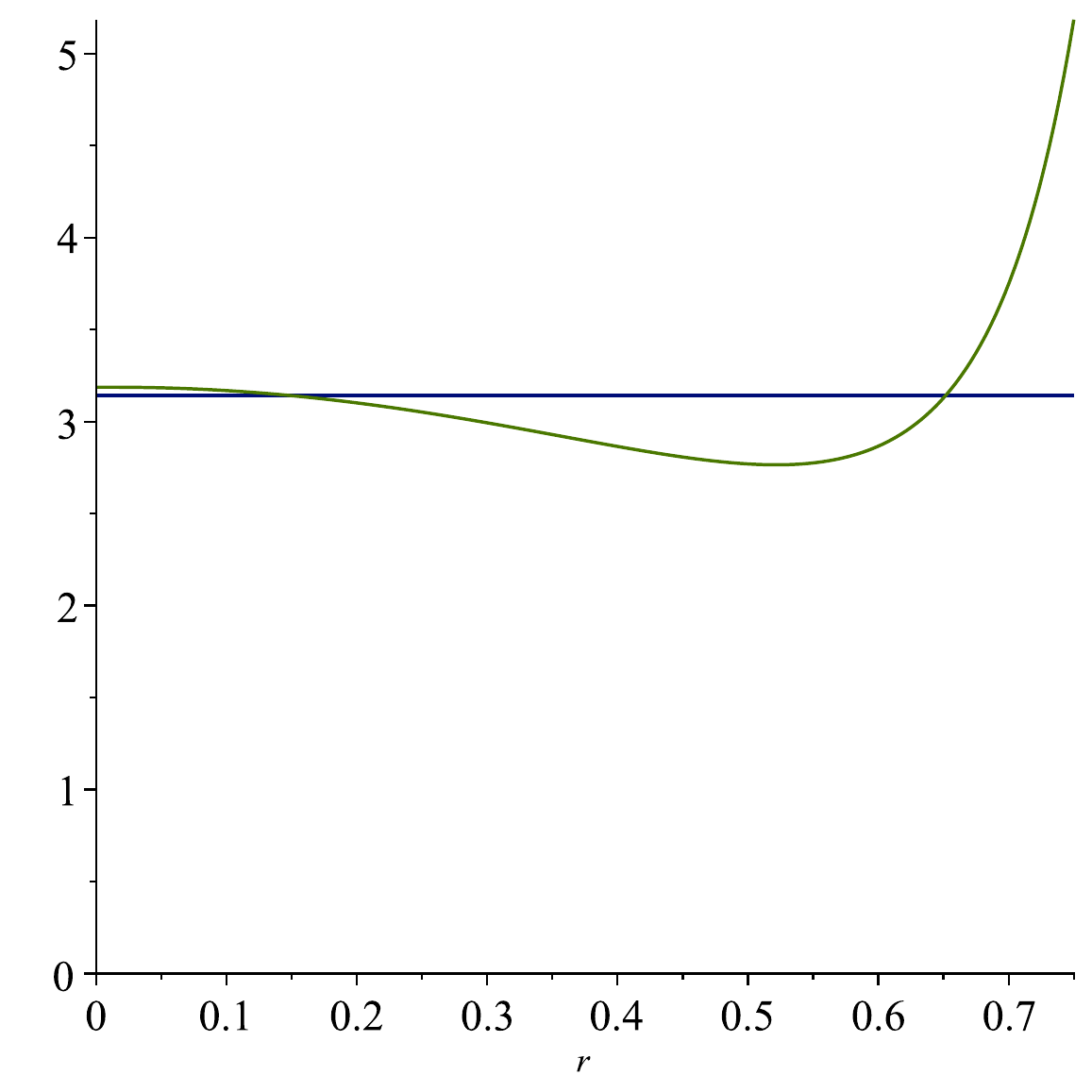}\hspace{0.3cm}
	\includegraphics[width=.3\textwidth]{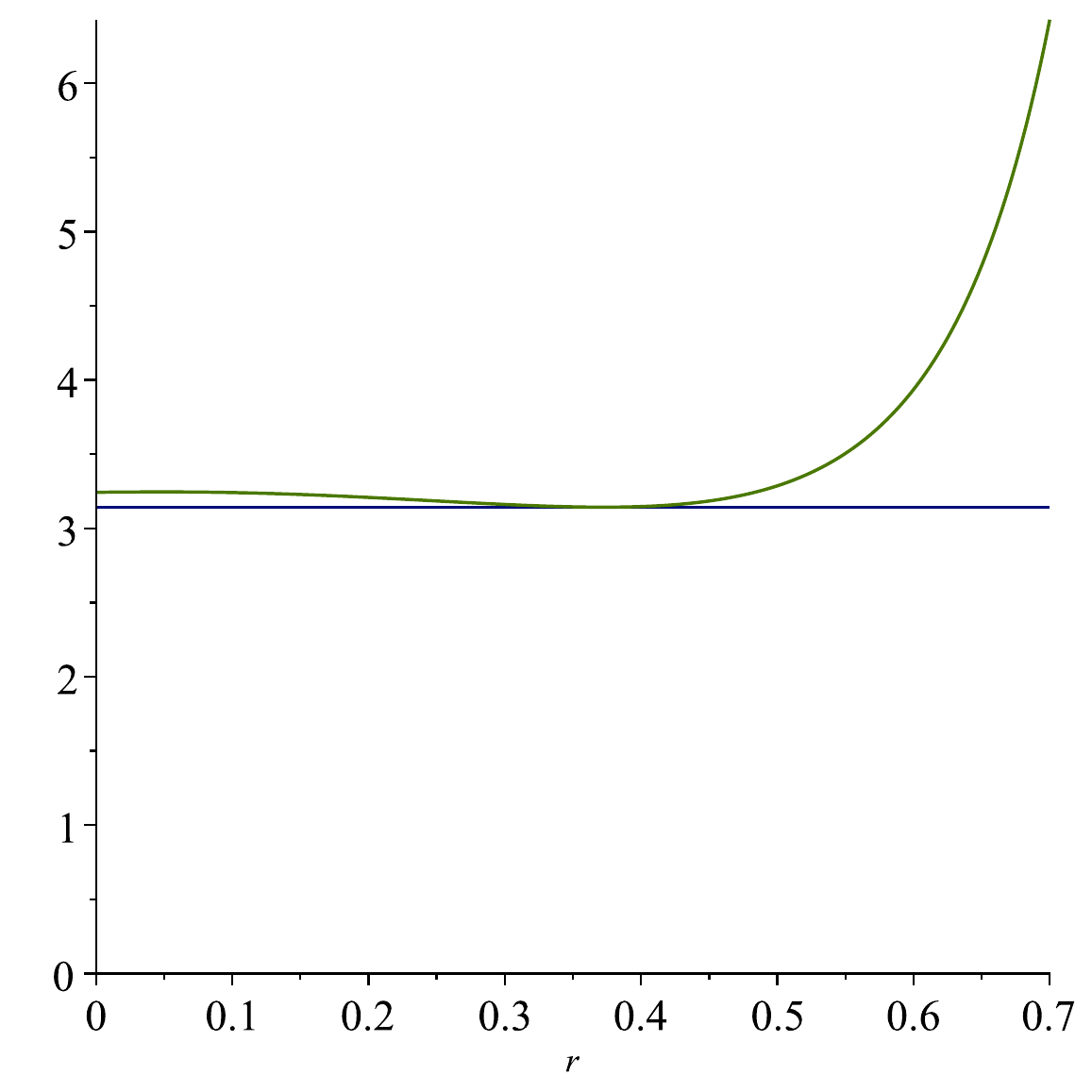}\hspace{0.3cm}
	\includegraphics[width=.3\textwidth]{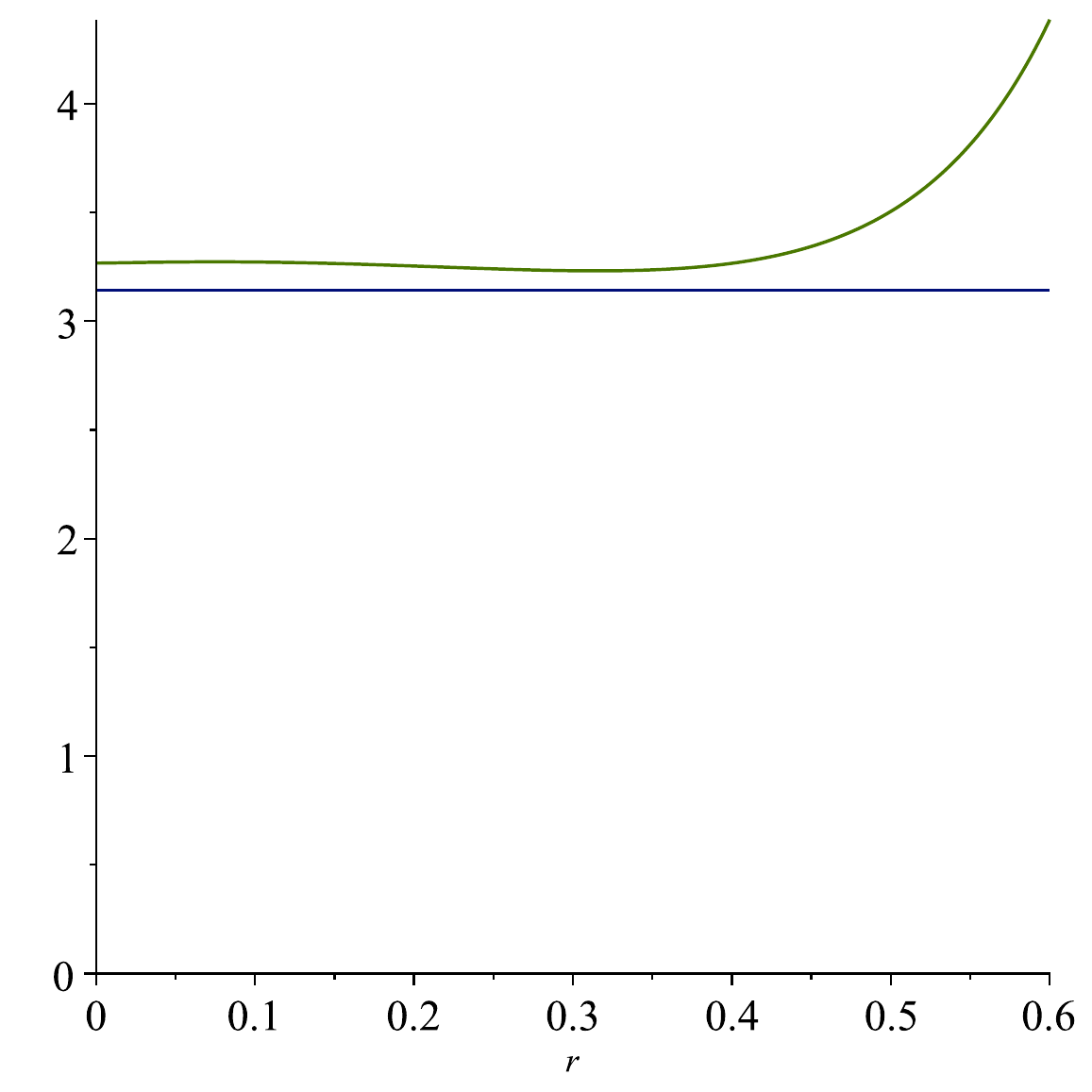} 
	\caption[] 
	{Plots of $g(0.03,\,.\,)$, $g(0.045,\,.\,)$, and $g(0.05,\,.\,)$
		(left to right). The straight line marks the threshold energy level $\pi$.}
	\label{fig:radial_minima_Dirichlet} 
\end{figure}

\begin{remark} 
	The nonconstant minimisers in the previous example display some qualitative properties, which we can show in any dimension.
Let  $u\in \mathscr{D}_{\operatorname{rad}}$ be a nonconstant radial minimiser on $B_1(0)\subset \mathbb{R}^n$, which exists for sufficiently small boundary data $u_0\equiv \textrm{const}>0$. 
[The existence can be seen as follows: Take a function $v\in C^\infty (B_1(0))$ with $v(x)\equiv 1$ close to $\partial B_1(0)$ and $|\operatorname{supp} (v)|< \frac{e_n}{2}$. Then $\lim_{u_0 \downarrow 0}F(u_0\cdot v)\le\frac{e_n}{2} $, while $F(x\mapsto u_0)\equiv e_n$].

We introduce now $\overline{B_\rho(0)}:= \{u=|\nabla u|=0\}$, $\Delta_\rho:=\lim_{|x|\downarrow \rho}\Delta u(x)$,  $\Delta_1:=\lim_{|x|\uparrow 1}\Delta u(x)$. Notice that these limits are well defined as $\Delta u \vert_{B_1(0)\setminus \overline{B_\rho(0)}}$ is a radial harmonic function.
We prove the following \\
\textbf{Claim.}
\begin{equation}\label{eq:rem10.1}
\Delta_\rho \cdot \Delta_1 <0.
\end{equation}
\textit{Proof of the claim.} Assume by contradiction that $\Delta_\rho \cdot \Delta_1\ge 0$. 
\\
We consider first the case, where $\Delta_\rho\ge 0$ and $ \Delta_1\ge 0$. 
Then $\Delta u\ge 0$ (i.e., $u$ is subharmonic) on $B_1(0)\setminus\overline{B_\rho(0)}$. First, the strong maximum principle shows that $ u < u_0$ on this annulus, and then Hopf's boundary point lemma (cf. \cite[Lemma 3.4]{GilbargTrudinger}) yields that
$\lim_{|x|\uparrow 1}\frac{\partial u}{\partial r} (x)>0$, a contradiction.\\
We consider next the case, where $\Delta_\rho\le 0$ and $ \Delta_1\le 0$. 
Then $\Delta u\le 0$ (i.e., $u$ is superharmonic) on $B_1(0)\setminus\overline{B_\rho(0)}$. First, the strong maximum principle shows that $ u > 0$ on this annulus, and then Hopf's boundary point lemma yields that
$\lim_{|x|\downarrow \rho}\frac{\partial u}{\partial r} (x)>0$, again a contradiction.
The claim is proved.

\medskip
\noindent
As a consequence we have in particular for nonconstant minimisers in $ \mathscr{D}_{\operatorname{rad}}$:
\begin{enumerate}
    \item \emph{{Discontinuity of the Laplacian.}} 
    $\lim_{|x|\uparrow \rho}\Delta u(x)=0 \not=\Delta_\rho=\lim_{|x|\downarrow \rho}\Delta u(x)$.
\item \emph{{Sign-change of the Laplacian}.} 
It is immediate from \eqref{eq:rem10.1} that $\Delta u$ changes sign on $B_1(0) \setminus \overline{B_\rho(0)}$.
This behaviour is in contrast 
to the situation of Navier boundary conditions, cf. Example \ref{eq:Naviern=2}.
\end{enumerate}
Harmonicity of $\Delta u$ shows further that it has in the annulus $B_1(0) \setminus \overline{B_\rho(0)}$ exactly one nodal hypersurface 
$\{|x|=\rho_0\}=\{x\in B_1(0) \setminus \overline{B_\rho(0)}: \Delta u(x) =0\}$
 for some suitable $\rho_0\in (\rho,1)$. In view of the 
 of the polar form of $\Delta$ for radial functions we find 
 $$
 \forall r\in (\rho,1):\quad u(x)|_{|x|=r} =\int^r_\rho s^{1-n}\int_\rho^s\sigma^{n-1} \Delta u(\sigma) \, \operatorname{d}\sigma\, \operatorname{d}s.
 $$
 Hopf's lemma yields then  as above  that $\Delta u >0$ on $B_{\rho_0}(0) \setminus \overline{B_\rho(0)}$ and $\Delta u <0 $ on $B_1(0) \setminus \overline{B_{\rho_0}(0)}$. Making again use of the formula above 
 we conclude that $u>0$ on $B_1(0) \setminus \overline{B_\rho(0)}$.
\end{remark}


\section{Radiality of Navier minimisers}
\label{sec:naviersym}

In this section we investigate radiality of Navier minimisers in $\mathscr{N}$ in the case that $\Omega = B_1(0)$  and $\varphi = u_0 \equiv \mathrm{const}$. Showing radiality of minimisers means that minimisers in $\mathscr{N}_{\operatorname{rad}}$ coincide with minimisers in $\mathscr{N}.$ This means in particular that the observations we have made for radial Navier minimisers (e.g. that their Laplace is never continuous on $B_1(0)$, cf. Remark \ref{rem:4}) are also observations about minimisers in $\mathscr{N}$.

In this section $\mathscr{N}(u_0)$ denotes the set $\{ u \in H^2(\Omega) : u-u_0 \in H_0^1(\Omega) \}$ for $u_0 > 0$.  We recall the notation $e_n = |B_1(0)|$ for the unit ball in $\mathbb{R}^n$. 

\begin{remark}\label{rem:5.1}
We want to investigate radiality of minimisers in the case of $\Omega = B_1(0)$ and $u_0 \equiv \mathrm{const}> 0$. In Lemma \ref{lem:energinfi} we have already shown that 
\begin{equation*}
    \inf_{\psi \in \mathscr{N}(u_0)} F(\psi) \leq e_n.
\end{equation*}
In the case that equality holds one already obtains trivially a radial minimiser, namely the constant solution $u (x)\equiv u_0$. Indeed, if equality holds then one computes
\begin{equation}
    F(u_0) = \int_{B_1(0)} ( \Delta u_0)^2 \; \mathrm{d}x + |\{ u_0 \neq 0 \}| = e_n = \inf_{\psi \in \mathscr{N}(u_0)} F(\psi).
\end{equation}
In this case it is an obvious guess that any minimiser should be radially symmetric, although according to Example~\ref{ex:1.1} there may also be nonconstant minimisers. However, one may wonder whether also in the generic case, where
$\inf_{\psi \in \mathscr{N}(u_0)} F(\psi) < e_n$, minimisers were always radial.

The following theorem answers this question in the affirmative.
\end{remark}

\begin{theorem}\label{thm:symmetry}
Let $\Omega = B_1(0)$, $u_0 \equiv \mathrm{const}> 0$ be a constant boundary datum 
and $u \in  \mathscr{N}(u_0)$ be any minimiser. 
Then $u$ is radially symmetric.
\end{theorem}

The following proof builds on the \emph{Talenti symmetrisation technique}, introduced by G. Talenti in \cite{Talenti}. This technique uses the \emph{symmetric decreasing rearrangement} of the function $\Delta u$, where $u$ is a minimiser. Our rearrangement works slightly differently, we want to rearrange functions so that they are defined on an \emph{annular region}. The reason for this is that the classical symmetric decreasing rearrangement would relocate the set $\{ \Delta u = 0 \}$ to the boundary of $B_1(0)$. However for minimisers this set is to be expected in the middle of $B_1(0)$, because it must contain (up to a null set) the flat set $\{ u = 0 \}$.    

Since we define different to \cite{Talenti} the rearrangement in annular regions, we have to argue in what way Talenti's symmetrisation results carry over. We therefore collect some properties of the used rearrangement procedure in Appendix \ref{app:rearrangement}. The proofs of these properties do not defer very much from the classical case. Nevertheless, we present them for the reader's convenience.

\begin{proof}[Proof of Theorem \ref{thm:symmetry}]
Let $u \in \mathscr{N}(u_0)$ be a minimiser. According to Corollary~\ref{cor:global_Sobolev_regularity}
we have $u\in C^{1,\alpha}(\overline{B_1(0)})$ and may conclude that
\begin{equation}\label{eq:main_assumption}
\Omega_0 := \{x\in \Omega : u(x) = 0 \} \subset \subset \Omega = B_1(0).
\end{equation}
Since $u$ is biharmonic in $B_1(0)\setminus \Omega_0$, elliptic boundary regularity yields that we even have 
$u\in C^{1,\alpha}(\overline{B_1(0)})\cap C^\infty (\overline{B_1(0)}\setminus \Omega_0)$.

Suppose that $u$ is nonradial. Let $r_0 \in [0,1]$ be such that $|\Omega_0| = e_n r_0^n$. We remark first that $r_0<1$ is obvious from  \eqref{eq:main_assumption}. We observe further that $r_0>0$. Assume by contradiction that $r_0=0$, i.e. $|\Omega_0| =0$ and hence $|\{u\not=0\}|=e_n$. In view of Lemma~\ref{lem:energinfi} (cf. also Remark~\ref{rem:5.1}) this implies that
$F(u)=e_n$ and hence that $\Delta u=0$ a.e. in ${B_1(0)}$. Since by assumption $u\in H^2({B_1(0)})$, elliptic regularity yields that $u$ is classically harmonic in ${B_1(0)}$ and $u(x)\equiv u_0$, a contradiction. To conclude, we have shown that
$$
r_0\in (0,1).
$$

 We now investigate the \emph{annular symmetric decreasing rearrangement} of $|\Delta u|\big\vert_{B_1(0) \setminus \Omega_0}$, i.e. the radial function $|\Delta u|^* : B_1(0) \setminus \overline{B_{r_0}(0)} \rightarrow [0,\infty)$, which is uniquely determined by 
    \begin{equation*}
    \forall t \geq 0:\quad 
        \{ x \in B_1(0) \setminus \overline{B_{r_0}(0)} : |\Delta u|^* > t \} = B_{r(t)}(0) \setminus \overline{B_{r_0}(0)} 
    \end{equation*}
    for some $r(t) \in [r_0, 1]$ chosen in such a way that   \begin{equation*}
    \forall t \geq 0: \quad 
    e_n\left(r(t)^n-r_0^n \right) = |\{ x \in B_1(0) \setminus \Omega_0 : |\Delta u| > t \}| .
    \end{equation*}
     For details on this rearrangement procedure we refer again to Appendix \ref{app:rearrangement}. In particular, we have
    \begin{equation*}
    \forall t \geq 0: \quad 
        |\{ x \in B_1(0) \setminus \overline{B_{r_0}(0)} : |\Delta u|^* > t \}| = |\{ x \in B_1(0) \setminus \Omega_0 : |\Delta u| > t \}| .
    \end{equation*}
    We define $f:= |\Delta u|^* \in L^2({B_1(0) {\setminus \overline{B_{r_0}(0)}} })$ and for $r \in (r_0,1)$ and for some arbitrary vector $\xi \in \mathbb{R}^n$ with $|\xi| \equiv 1$ we set $\tilde{f}(r) := f(r\xi)$. Notice that this definition is independent of $\xi$, because $f$ is radial. 
    Next we define a function $w :[0,1] \rightarrow \mathbb{R}$ by
    \begin{equation}\label{eq:definition_of_w}
    w(r):= \int^r_{r_0} \rho^{1-n}\int^\rho_{r_0}\sigma^{n-1}\tilde{f}(\sigma)\,\mathrm{d}\sigma\, \mathrm{d} \rho, 
    \end{equation}
    such that $w \big\vert_{[0,r_0)} \equiv 0$ and $w \big\vert_{[r_0,1]}$ is the unique solution of 
    \begin{equation}\label{eq:laplaceradial}
        \begin{cases}
            w''(r) + \frac{n-1}{r}w'(r) = \tilde{f}(r), & \quad r \in (r_0,1], \\
             w(r_0) = w'(r_0) = 0.
        \end{cases}
    \end{equation}
    Moreover we define $v(x) := w(|x|)$. Since the left- and right-sided limits of $w$ and of $w'$ match up at $r = r_0$ one checks readily that $v \in H^2(B_1(0))$ and therefore $v \in \mathscr{N}_{\operatorname{rad}}(w(1))$, i.e. it is admissible for the minimisation in $\mathscr{N}_{\operatorname{rad}}(w(1))$.
    Notice that $\Delta v = f $ a.e. in $B_1(0) \setminus \overline{B_{r_0}(0)}$ since the left hand side in \eqref{eq:laplaceradial} is exactly the Laplace operator for radial functions. 
    Therefore one has 
    \begin{equation}\label{eq:deltav}
        \Delta v(x) = \begin{cases}
            f(x), & \mbox{\ if\ } r_0 <|x|<1, \\ 
            0,  & \mbox{\ if\ } |x|\leq r_0 .
        \end{cases}
    \end{equation}
    The main ingredients of the proof will be the following claims. \\
    \textbf{Claim 1.} $F(v) = F(u)$. \\
    \textbf{Claim 2.} $u_0 \leq w(1)$. \\
    \textbf{Claim 3.} Since by assumption $u$ is nonradial, one even has $u_0 < w(1)$. 
    \\
    If Claims 1, 2 and 3 are shown one can finish the proof. Indeed, employing the (strict) monotonicity properties proven in Remark \ref{rem:strictlyincr} we find:
    \begin{equation*}
        F(v) \geq \inf_{ \psi \in \mathscr{N}_{\operatorname{rad}}(w(1)) } F(\psi) \stackrel{(*)}{\ge}
        \inf_{ \psi \in \mathscr{N}_{\operatorname{rad}}(u_0) }  F(\psi) \geq \inf_{ \psi \in \mathscr{N}(u_0) }  F(\psi) = F(u) = F(v).
    \end{equation*}
    This shows that $v$ minimises $F$ in $\mathscr{N}_{\operatorname{rad}}(w(1))$. Since $v$ is not a constant, the inequality sign $(*)$ is even strict, again according to Remark \ref{rem:strictlyincr}.
    This  contradiction shows that each minimiser must be radial.
    
     \bigskip
    \noindent
    \textit{Proof of Claim 1.} 
We recall that by Lemma \ref{lem:Stam} $\Delta u = 0$ a.e. on $\Omega_0=\{ u = 0\}$. Observing \eqref{eq:deltav} and making use of Lemma~\ref{lem:rearrineq}(a) we find
\begin{align}
    \int_{B_1(0)} ( \Delta u )^2 \; \mathrm{d}x  &
     = \int_{B_1(0) \setminus \Omega_0} |\Delta u|^2 \; \mathrm{d}x
     = \int_{B_1(0) \setminus \overline{B_{r_0}(0)}} f(x)^2 \; \mathrm{d}x \nonumber\\
     &= \int_{B_1(0) \setminus \overline{B_{r_0}(0)}} (\Delta v(x))^2 \; \mathrm{d}x = \int_{B_1(0)} (\Delta v)^2 \; \mathrm{d}x.\label{eq:UnglBiharmenerg}
\end{align}
Next notice that by the definition of $w$ (cf. \eqref{eq:laplaceradial}) one has 
$\{ v = 0 \} \supset \overline{B_{r_0}(0)}$ and therefore $\{v \neq 0 \} \subset B_1(0) \setminus \overline{B_{r_0}(0)}$. 
Moreover,  $f=|\Delta u|^*\not\equiv 0$ on $B_1(0)\setminus \overline{B_{r_0}(0)}$
because otherwise $u$ would be harmonic in $B_1(0)$ and so $u>0$ in view of $u_0>0$, which contradicts $|\Omega_0| >0$. 
 We conclude that 
 \begin{equation}\label{eq:f_strictly_positive}
  \tilde f\ge 0, \quad \mbox{and}\quad
 \tilde f >0\quad  \mbox{ at least on an interval } \quad (r_0,r_0+\varepsilon)
 \quad \mbox{  for some } \quad \varepsilon >0. 
 \end{equation} 
 {From} \eqref{eq:definition_of_w} we see that $w(r)>0$ for all $r\in (r_0,1]$. 
This yields that even 
$\{ v = 0 \} = \overline{B_{r_0}(0)}$ and therefore $\{v \neq 0 \} = B_1(0) \setminus \overline{B_{r_0}(0)}$. 
Hence one can compute 
\begin{equation*}
    |\{ v \neq 0 \}|= |B_1(0) \setminus \overline{B_{r_0}(0)}| = |B_1(0)| - e_n r_0^n  = |B_1(0)  |-  |\Omega_0| = |B_1(0) \setminus \Omega_0| = |\{ u \neq 0 \}|. 
\end{equation*}
This and \eqref{eq:UnglBiharmenerg} imply that $F(u) = F(v)$.

    \bigskip
    \noindent
    \textit{Proof of Claim 2.} We define for $t \geq 0$ the increasing function 
    \begin{equation*}
        \mu(t) := |\{ x \in B_1(0) : u(x) < t \}|. 
    \end{equation*}
    For short we also use the notation $\mu(t) = |\{ u < t \}|$. 
    Assume now that $t\in (0,u_0)$ is a regular value of $u$.
     In view of the smoothness $u\in C^\infty (\overline{B_1(0)}\setminus \Omega_0)$  this is true by Sard's theorem for a.e. $t$. Then $\{ u < t \} \subset \subset {B_1(0)}$ is a $C^1$-smooth domain with boundary $\{ u = t \}$ and outward unit normal $\nu = \frac{\nabla u}{|\nabla u|}$, see \cite[3.4.3]{Federer}. Hence the Gauss divergence theorem yields 
    \begin{equation}\label{eq:equal1}
        \int_{\{ u < t \}} |\Delta u| \; \mathrm{d}x \geq \int_{\{ u < t \} } (\Delta u ) \; \mathrm{d}x = \int_{\{ u = t \} } \nabla u \cdot \nu  \; \mathrm{d}\mathcal{H}^{n-1} = \int_{\{ u = t \} } |\nabla u | \; \mathrm{d} \mathcal{H}^{n-1}.
    \end{equation}
     According to \cite[Kap. VIII, \S 2, Satz 4]{Natanson}, 
     $\mu$ is differentiable in almost all $t$ because $\mu$ is increasing. For such $t$ we infer from the coarea formula (see e.g. \cite[Section 3.3.4, Proposition 3]{EvGar}) that 
    \begin{equation*}
        \mu'(t) = \int_{\{ u = t \} } \frac{1}{|\nabla u|} \; \mathrm{d}\mathcal{H}^{n-1}.
    \end{equation*}
    We infer from this, the Cauchy Schwarz inequality and the isoperimetric inequality that the following holds for a.e. $t \in (0, u_0)$:
    \begin{align}\label{eq:isop}
         \mu'(t) \int_{\{u < t \}} |\Delta u |\; \mathrm{d}x  & 
         \geq  \left( \int_{\{ u = t \} } \frac{1}{|\nabla u|} \; \mathrm{d}\mathcal{H}^{n-1}\right) \, \left(\int_{\{ u = t \} } |\nabla u | \; \mathrm{d} \mathcal{H}^{n-1}\right)\nonumber\\
       & \geq \left( \int_{\{ u = t \} }
 |\nabla u|^{\frac{1}{2}} \frac{1}{|\nabla u|^{\frac{1}{2}}} \; \mathrm{d}\mathcal{H}^{n-1} \right)^2       
  = \mathcal{H}^{n-1} (\{  u = t \})^2\nonumber\\
  & \geq ( n e_n^{1/n} |\{ u < t \}|^\frac{n-1}{n} )^2  = n^2 e_n^{2/n} \mu(t)^{(2n-2)/n}.
    \end{align}
    Now we estimate with Lemma \ref{lem:Stam}, the definition of $\Omega_0=\{u=0\}$ and Lemma \ref{lem:rearrineq}(b)
    \begin{align}\label{eq:sublevellap}
        \int_{ \{u < t \} } |\Delta u| \; \mathrm{d}x & = \int_{ \{ u < t \} \setminus \Omega_0} |\Delta u |\; \mathrm{d}x = \int_{B_1(0) \setminus \Omega_0} 1_{\{  u<t\} \setminus \Omega_0} |\Delta u | \; \mathrm{d}x \\ &  \leq \int_{B_1(0) \setminus \overline{B_{r_0}(0)}} (1_{\{u<t\}\setminus \Omega_0})^* |\Delta u|^* \; \mathrm{d}x = \int_{B_1(0) \setminus \overline{B_{r_0}(0)}} (1_{\{u<t\}\setminus \Omega_0})^* f \; \mathrm{d}x. \nonumber
    \end{align}
    By Example \ref{ex:charfuncstar} one has that 
    \begin{equation*}
        (1_{\{ u < t \} \setminus \Omega_0 })^* = 1_{B_{r(t)}(0) \setminus \overline{B_{r_0}(0)}},  \quad \textrm{where $r(t) \in [r_0,1]$ is such that}
    \end{equation*}
    \begin{equation*}
        e_n r(t)^n - e_n r_0^n = |\{ x \in B_1(0) \setminus \Omega_0 : u(x) < t \}| =|\{ x \in B_1(0)  : u(x) < t \}| -|\Omega_0| = \mu(t) - e_n r_0^n,
    \end{equation*}
    i.e. $r(t) = ( \frac{\mu(t)}{e_n} )^{1/n}$. Using this and layer cake integration in \eqref{eq:sublevellap} we find \begin{equation*}
        \int_{\{ u < t \}} |\Delta u | \; \mathrm{d}x \leq \int_{r_0}^{r(t)} n e_n r^{n-1} \tilde{f}(r) \; \mathrm{d}r = \int_{r_0}^{(\frac{\mu(t)}{e_n})^{1/n}} n e_n r^{n-1} \tilde{f}(r) \; \mathrm{d}r.
    \end{equation*} 
    Plugging this into \eqref{eq:isop} we infer that
    \begin{equation*}
        \mu'(t) \int_{r_0}^{(\frac{\mu(t)}{e_n})^{1/n}} n e_n r^{n-1} \tilde{f}(r) \; \mathrm{d}r  \geq n^2 e_n^{2/n} \mu(t)^{(2n-2)/n}.
    \end{equation*}
    As a consequence we have  for a.e. $t \in (0,u_0)$ that
    \begin{equation}\label{eq:48}
        h(\tau) := \frac{e_n^{1-2/n}}{n \tau^{(2n-2)/n}} \int_{r_0}^{(\tau/e_n)^{1/n}} r^{n-1} \tilde{f}(r) \; \mathrm{d}r  \quad \textrm{satisfies} \quad h(\mu(t)) \mu'(t) \geq 1  .
    \end{equation}
    Since $h$ is in view of \eqref{eq:f_strictly_positive} strictly positive and continuous on 
    $(e_n r_0^n,\infty)$ it possesses an increasing primitive  on $(e_n r_0^n,\infty)$, which we call $H$. Integrating \eqref{eq:48} on $(\varepsilon,u_0)$ we find according to \cite[Kap. VIII, \S 2, Satz 5]{Natanson} that
    \begin{equation*}
        u_0 - \varepsilon \leq \int_{\varepsilon}^{u_0} h(\mu(t)) \mu'(t) \; \mathrm{d}t = \int_\varepsilon^{u_0} \frac{d}{dt} (H \circ \mu)(t) \; \mathrm{d}t  \leq (H \circ \mu)(u_0) - (H\circ \mu)(\varepsilon) = \int_{\mu(\varepsilon)}^{\mu(u_0)} h(\tau) \; \mathrm{d}\tau. 
    \end{equation*}
    Notice that the second inequality sign is due to the fact that $H\circ \mu$  is in general merely increasing and not necessarily absolutely continuous.  Using that $\mu(u_0) \leq e_n$, 
    $$\mu(\varepsilon) \geq e_n r_0^n+|\{u<0\}| \quad \mbox{ for all }\quad \varepsilon > 0
    $$ and letting $\varepsilon \rightarrow 0+$ we find 
    \begin{equation}\label{eq:equal2}
        u_0 \leq \int_{e_n r_0^n+|\{u<0\}| }^{\mu (u_0)} h(\tau) \; \mathrm{d}\tau
       \leq \int_{e_n r_0^n+|\{u<0\}|}^{e_n} h(\tau) \; \mathrm{d}\tau
        \leq \int_{e_n r_0^n}^{e_n} h(\tau) \; \mathrm{d}\tau.
    \end{equation}
    In view of the strict positivity of $h$ this inequality is strict in case
    that
    \begin{equation}\label{eq:u_positive}
    |\{u<0\}|=0
    \end{equation}
    is violated.
    Employing the substitution $\tau = e_n \sigma^n$, using \eqref{eq:48} and observing that for all $r \in [r_0,1]$ one has (by \eqref{eq:laplaceradial}) that $\frac{d}{dr} (r^{n-1} w'(r)) = r^{n-1} \tilde{f}(r)$ and $w(r_0) =w'(r_0) =0$ we infer
    \begin{align*}
        u_0 & \leq \int_{r_0}^1 n e_n \sigma^{n-1} h(e_n \sigma^n) \; \mathrm{d}\sigma 
        = \int_{r_0}^1 \frac{1}{\sigma^{n-1}} \int_{r_0}^\sigma r^{n-1}\tilde{f}(r) \; \mathrm{d}r
        \\ & = \int_{r_0}^1 \frac{1}{\sigma^{n-1}} \int_{r_0}^\sigma \frac{d}{dr} (r^{n-1} w'(r)) \; \mathrm{d}r
        = \int_{r_0}^1 w'(\sigma) \; \mathrm{d}\sigma = w(1) - w(r_0) = w(1) .
    \end{align*}
    Again the inequality is strict, if \eqref{eq:u_positive} is violated.
    
    We infer that $u_0 \leq w(1)$, where the inequality is strict if $|\{u<0\}|>0$.

\bigskip
\noindent
\textit{Proof of Claim 3.}   
We recall that we have $u\in C^{1,\alpha}(\overline{B_1(0)})\cap C^\infty (\overline{B_1(0)}\setminus \Omega_0)$. In case that $|\{u<0\}|>0$ we already proved that then $u_0<w(1)$ so that in what follows we may assume \eqref{eq:u_positive}, i.e. that
$$
 u\ge 0 \quad a.e. \mbox{\ in\ } B_1 (0).
$$
 Since Claim 2 is already shown in it suffices to prove that $u_0 = w(1)$ would already imply that $u$ is radial. 
 This part of the proof is based on the ideas in \cite{Lions}. In case of equality all estimates in the proof of Claim 2 turn out to be equalities (for a.e. $t \in (0,u_0]$). 
 
 We first claim that for all $t \in (0,u_0]$ one has that $\{ u < t \}$ is a ball. Equality (for a.e. $t$) in the isoperimetric inequality in \eqref{eq:isop} implies that $\{ u < t \}$ is a ball for every $t \in (0,u_0]\setminus N$ (for a null set $N$). For $t \in N$ one can however find an increasing sequence $(t_k)_{k \in \mathbb{N}} \subset (0,u_0] \setminus N$ such that $t_k \uparrow t$. We infer that 
    \begin{equation*}
        \{ u < t \} = \bigcup_{ k \in \mathbb{N}} \{ u < t_k \},
    \end{equation*}
    from which it follows that $\{ u < t \}$ is a ball, because the sequence of balls is increasing and 
    the sequences of centres and radii have convergent subsequences. However, we have not yet shown that these balls have the same centre.
    
    We next show that $\{ u < u_0 \} = B_1(0)$. Note that it suffices to show (since $\{u<u_0\}$ is a ball) that $|\{ u < u_0 \}| = |B_1(0)|$, i.e. $\mu(u_0) = e_n$.
    Equality in \eqref{eq:equal2} yields that $h \equiv 0$ on $(\mu(u_0),e_n)$.  If we assume now that $\mu(u_0)< e_n$ then $h \equiv 0$ on an open interval with upper bound $e_n$. This and \eqref{eq:48} would yield
    \begin{equation*}
       0 = \lim_{z \rightarrow e_n} h(z) = \frac{1}{ne_n} \int_{r_0}^1 r^{n-1} \tilde{f}(r) \; \mathrm{d}r,
    \end{equation*}
    which is impossible in view of \eqref{eq:f_strictly_positive}.
    
     Equality in \eqref{eq:equal1} for a.e. $t\in (0,u_0]$ yields that $\Delta u \geq 0$ a.e. in $\{ u < u_0\}$. Hence, 
     $u$ is subharmonic on $B_1(0)$.
    Thereupon the Hopf boundary lemma for subharmonic functions  (cf. \cite[Lemma 3.4]{GilbargTrudinger} yields that for all $t \in (0,u_0]$ one has  $\nabla u \neq 0$ on $\partial \{ u < t \}$. Notice here that the special case of $t= u_0$ uses the fact that $u \in C^{1,\alpha}(\overline{B_1(0)})$.  
    We also remark that it is needed here that $\{ u < t \}$ is a ball, otherwise it would not necessarily satisfy the required interior ball condition. We infer that for all $t \in (0,u_0]$ one has $\partial \{ u < t \} = \{ u = t \}$ and $|\nabla u | > 0$ on $\{ u = t \}$. This makes each $t\in (0,u_0]$ a regular value of $u$ and in particular the computation in \eqref{eq:equal1} can be done for all $t \in (0,u_0]$ (and not just for a.e. $t$).
    However, notice that equality in \eqref{eq:isop} still may be violated for $t$ in a null set of $(0,u_0)$. Equality in the application of the Cauchy-Schwarz-inequality yields that for a.e. $t> 0$ one has $|\nabla u| \equiv \mathrm{const} > 0$ on $\{ u = t \}$. Since for each $t \in (0,u_0]$ the level set $\{  u = t \}$ is regular, it can be approximated by the level sets $(\{ u = s \})_{s \in (t-\varepsilon,t)}$ and therefore $|\nabla u | \equiv \mathrm{const}$ does not just hold on almost every level set $\{ u = t \}$ but on every level set.
    This information given we now show that all balls $\{ u < t \}$, $t \in (0,u_0]$ have the same centre.
    We can define a continuous function $\eta : (0,u_0] \rightarrow \mathbb{R}$ that satisfies $|\nabla u|(x) =  \eta(u(x))$. (Notice that the continuity of $\eta$ follows again from the fact that each value is regular). We also set for $x \in \overline{B_1(0)}\setminus \Omega_0 $ 
    \begin{equation*}
        \phi(x) := -\int^{u_0}_{u(x)} \frac{1}{\eta(\tau)} \; \mathrm{d}\tau.
    \end{equation*}
    Now for each $z \in \partial B_1(0)$ we look at a (a priori not necessarily unique) non-extendable solution $x_z:J_z\to \mathbb{R}^n\setminus \Omega_0$ of the differential equation 
    \begin{equation}\label{eq:diffeqx}
        \begin{cases}
             x_z'(s) =  \frac{\nabla u}{|\nabla u|}(x_z(s)) & (s\in J_z),  \\ \; \; \;   x_z(0) = z.
        \end{cases}
    \end{equation}
    Here $J_z \subset \mathbb{R}$ is the open maximal interval of existence and contains the point $0$.    The existence of such a solution follows from Peano's theorem and Zorn's lemma, see e.g. \cite[Theorem 3.22]{Philip}. 
   Outside  of $\overline{B_1(0)}$ we have to choose a Hölder continuous extension of $\frac{\nabla u}{|\nabla u|}$. The chosen extension will, however, not be relevant for our argument as we will only look at values $s$ such that $x_z(s) \in \overline{B_1(0)}$.  
   Notice that for all $s \leq 0$ one has that $x_z(s) \in B_1(0)$ since $B_1(0) = \{ u < u_0 \}$ and $\frac{d}{ds} u(x_z(s)) = |\nabla u|(x_z(s)) > 0$ for all $s \in J_z, s\leq 0$. {From} now on let $\tilde{J}_z = J_z \cap (-\infty,0]$.    
   We compute for $s \in \tilde{J}_z$ that
    \begin{equation}
        \frac{d}{ds} \phi(x_z(s)) = \frac{1}{\eta(u(x_z(s)))} \nabla u(x_z(s)) \cdot x_z'(s) = \frac{1}{|\nabla u|(x_z(s))} \nabla u(x_z(s)) \cdot x_z'(s)  = 1. 
    \end{equation}
    Notice that for all $s_1, s_2 \in \tilde{J}_z$  one has 
    \begin{align*}
        |x_z(s_1) - x_z(s_2)| & = \left\vert \int_{s_1}^{s_2} x_z'(s) \; \mathrm{d}s \right\vert \leq\left\vert \int_{s_1}^{s_2}|x_z'(s)| \; \mathrm{d}s\right\vert = |s_1 - s_2|   \\ & = \left\vert \int_{s_1}^{s_2} \frac{d}{ds} \phi(x_z(s)) \; \mathrm{d}s \right\vert  = |\phi(x_z(s_1)) - \phi(x_z(s_2))| \leq |x_z(s_1) - x_z(s_2)|,
    \end{align*}
    where we used in the last step that $|\nabla \phi(x)| = \frac{1}{\eta(u(x))}|\nabla u(x)| = 1$ for all $x \in B_1(0)\setminus \Omega_0$. We conclude that for all $s_1,s_2 \in \tilde{J}_z$ one has $|x_z(s_1)-x_z(s_2)| = |s_1 -s_2|$. By virtue of the triangle inequality  this implies that $x_z\vert_{\tilde{J}_z}$ is a straight line and is {moreover} parameterised with unit speed. Therefore there exists $v \in \mathbb{R}^n$ such that $|v| = 1$ and $x_z(t) = z + tv$ for all $t \in \tilde{J}_z$.
    Looking at \eqref{eq:diffeqx} at $s= 0$ we infer also that
    \begin{equation}\label{eq:46}
        v =  x_z'(0) \big\vert_{s = 0} = \frac{\nabla u}{|\nabla u|}(x_z(0)) =  \frac{\nabla u}{|\nabla u|}(z)   = \nu_{\{u= u_0\}}(z) = \nu_{\partial B_1(0)}(z) = z. 
    \end{equation}
    Hence $x_z(s) = (1+s)z$ and for all $s \in \tilde{J}_z$ and (since $|\nabla u|>0$ on $\{ u \neq 0 \}$) maximality of $\tilde{J}_z$ yields $\tilde{J}_z =  \{ s \in (-1,0) : u((1+s)z) \neq 0 \}$. 
For each $z \in \partial B_1(0)$ we define  $f_z: \tilde{J}_z \rightarrow \mathbb{R}, s \mapsto u((1+s) z)$. To prove radiality it suffices to show that $f_z$ does not depend on $z$. 
    For all $s \in \tilde{J}_z$ one has
    \begin{align}\label{eq:46.5}
       \frac{d}{ds} u((1+s) z) & = \frac{d}{ds} u(z + s z ) = \frac{d}{ds} u(x_z(s)) = \nabla u (x_z(s)) x_z'(s) \\ & = |\nabla u(x_z(s))| = \eta (u(x_z(s))) = \eta( u((1+s)z)) .
    \end{align}
     Therefore for any $z$, $f_z$ is a solution of 
    \begin{equation*}
        \begin{cases}
             f_z'(s) = \eta(f_z(s)) & (s \in \tilde{J}_z), \\ 
             f_z(0) = u_0.
        \end{cases}
    \end{equation*}
    As $\eta$ is not Lipschitz continuous, no general uniqueness result is applicable to show that $f_z$ is independent of $z$. However, we can still obtain uniqueness by separation of variables. Indeed, If $\bar{\eta}$ is a primitive  of $\frac{1}{\eta}$ we obtain that 
    \begin{equation*}
        \frac{d}{ds} \bar{\eta} (f_z(s)) = \frac{f_z'(s)}{\eta(f_z(s))} = 1 \quad \Rightarrow \quad \forall s \in \tilde{J}_z  : \quad
        \bar{\eta}(f_z(s)) - \bar{\eta}(u_0) = s  .
    \end{equation*}
    Using that $\bar{\eta}$ is strictly increasing and hence invertible we obtain that 
    \begin{equation*}
        f_z(s) = \bar{\eta}^{-1}( \bar{\eta}(u_0) + s )  \quad (s \in \tilde{J}_z). 
    \end{equation*}
    We infer that $f_z$ (and also $\tilde{J}_z$, since $\tilde{J}_z= \{ f_z > 0 \}$), are independent of $z$. Hence $z \mapsto u((1+s) z)$ is independent of $z$ on the annulus $ \{ u > 0 \}$. We infer that $u$ is radial and $\{ u > 0\}$ is an annulus with centre zero. The radiality is shown. This completes the proof of Claim 3 and also the proof of the theorem. 
    \end{proof}

As a consequence we obtain the optimality of the $C^{1,\alpha}$-regularity of minimisers, which we already mentioned in the introduction.

\begin{corollary}\label{cor:opt_reg}
There exists a smooth domain ${B_1(0)} \subset \mathbb{R}^n$ and a smooth boundary datum $\varphi \in C^\infty(\overline{\Omega})$ such that each Navier minimiser $u \in \mathscr{N}$ does not lie in $C^2(\Omega)$.
\end{corollary}
\begin{proof}
    Let $n= 2$, $\Omega = B_1(0) \subset\mathbb{R}^2$ and $\varphi\equiv \mathrm{const} = u_0>0 $.
     By Example~\ref{eq:Naviern=2} it is clearly possible to choose $u_0$ so small such that 
    \begin{equation*}
    \inf_{v \in \mathscr{N}_{\operatorname{rad}}} F(v) < \pi
    \quad \Rightarrow\quad 
    \inf_{v \in \mathscr{N}} F(v) < \pi.
    \end{equation*}
    This implies that for any minimiser $u\in \mathscr{N}$ the flat set $\Omega_0=\{x\in \Omega :u(x)=0\}$ is not empty.
    
    The previous theorem yields that each minimiser is radial. Since nonconstant minimisers in $\mathscr{N}_{\operatorname{rad}}$ are not in $C^2$ by Remark \ref{rem:4} we infer that no minimiser can lie in $C^2$. 
\end{proof}

\begin{remark}
 This regularity behaviour is fundamentally different from the Alt-Caffarelli-type problem in \cite{DiKaVa}, where $C^2$-regularity can be expected and has already been proved in the case of $n = 2$ in \cite{Mueller}.  
\end{remark}


\appendix

\section{Proof of the local $BMO$-estimate}\label{sec:app_proof_BMO}

\begin{proof}[Proof of Theorem~\ref{thm:Delta_BMO}] We choose some $R_0\in (0,\frac13 r_\Omega)$ and keep it fixed in what follows.
	Let $0<r<R\le R_0$ be arbitrary. 
	We pick an arbitrary admissible $x_0\in \Omega$ (i.e. $\overline{B_{2R}(x_0)}\subset\subset B_{3R}(x_0) \subset \Omega$) and keep it fixed in what follows. By minimising $v\mapsto \int_{B_{2R}(x_0)}(\Delta v)^2\, \mathrm{d}x$ on $u+H^2_0 (B_{2R}(x_0))$ we find 
	$$
	\tilde h \in u+H^2_0 (B_{2R}(x_0)):\quad \Delta^2 \tilde h=0 \quad \mbox{weakly in }\quad B_{2R}(x_0).
	$$
	By elliptic regularity $\tilde{h}\in C^\infty(B_{2R}(x_0))$. We define $h \in \mathscr{N}$ (or respectively $h \in \mathscr{D}$) via
	$$
	h(x):=
	\begin{cases}
	\tilde h(x) \quad &\quad \mbox{for}\quad x\in B_{2R}(x_0),\\
	u(x) \quad &\quad \mbox{for}\quad x\in \Omega \setminus B_{2R}(x_0).
	\end{cases}
	$$
	By minimality of $u$ we have that $F(u)\le F(h)$ which implies that
	$$
	\int_{B_{2R}(x_0)}(\Delta u)^2\, \mathrm{d}x +|\{x\in B_{2R}(x_0):u(x)\not=0 \}|\le 
	\int_{B_{2R}(x_0)}(\Delta h)^2\, \mathrm{d}x +|\{x\in B_{2R}(x_0):h(x)\not=0 \}|.
	$$
	This yields
	\begin{equation}\label{eq:2.2}
	\int_{B_{2R}(x_0)}\left( (\Delta u)^2- (\Delta h)^2\right)\, \mathrm{d}x\le C R^n
	\end{equation}
	with a universal constant $C>0$. Since $\Delta^2 h=0$ in $B_{2R}(x_0)$ 
	and $u-h\in H^2_0 (B_{2R}(x_0))$ we observe that
	\begin{align*}
	\int_{B_{2R}(x_0)}\left( (\Delta u)^2- (\Delta h)^2\right)\, \mathrm{d}x
	=& \int_{B_{2R}(x_0)}\left( \Delta u-\Delta h\right)\,\left( \Delta u+\Delta h\right)\, \mathrm{d}x\\
	=& \int_{B_{2R}(x_0)}\left( \Delta u-\Delta h\right)\,\left( \Delta u-\Delta h\right) \, \mathrm{d}x
	=  \int_{B_{2R}(x_0)}\left( \Delta u-\Delta h\right)^2 \, \mathrm{d}x.
	\end{align*}
	With this, we conclude from \eqref{eq:2.2} that 
	\begin{equation}\label{eq:2.3}
	\int_{B_{2R}(x_0)} \left( \Delta u-\Delta h\right)^2 \, \mathrm{d}x\le C R^n.
	\end{equation}
	By means of H\"older's inequality
	$$
	|(\Delta u)_{x_0,r} - (\Delta h)_{x_0,r} |^2\le \left(\fint_{B_r(x_0)} |\Delta u -\Delta h|\, \mathrm{d}x\right)^2 \le \fint_{B_r(x_0)} |\Delta u -\Delta h|^2\, \mathrm{d}x,
	$$
	we conclude from \eqref{eq:2.3} that 
	\begin{equation}\label{eq:2.4}
	\int_{B_r(x_0)}|(\Delta u)_{x_0,r} - (\Delta h)_{x_0,r} |^2\, \mathrm{d}x\le \int_{B_r(x_0)} |\Delta u -\Delta h|^2\, \mathrm{d}x
	\le C R^n.
	\end{equation}
    From here on one may conclude the proof by putting $U:=\Delta u$ und $H:= \Delta h$ and then copying word by word the proof
    of Theorem~\ref{thm:app_Delta_BMO}, starting with the line below \eqref{eq:app_2.4}.
\end{proof}

\section{Symmetric decreasing rearrangements in domains with holes}\label{app:rearrangement}

In the present article, we use a special type of rearrangement, namely a symmetric rearrangement in \emph{annular regions}. In this appendix we introduce the rearrangement procedure and discuss all important needed properties. 

A fundamental observation that is needed throughout this appendix is that a function $f: A \rightarrow [0,\infty)$ (A arbitrary set) is uniquely determined by its \emph{superlevel sets}, i.e. sets of the form $\{ x \in A : f(x) > t \}$, $t \geq 0$, which we call for short $\{ f > t \}$ in the sequel, if $A$ is clear from the context. Indeed, once all superlevel sets $\{ f > t \}$ are known $f$ can be retrieved by means of the formula 
\begin{equation}\label{eq:fdarst}
    f(x)= \int_0^{f(x)} 1 \; \mathrm{d}t = \int_0^\infty 1_{\{ f> t \}}(x) \; \mathrm{d}t.
\end{equation}

We recall the notation $e_n := |B_1(0)|$. Throughout this section we let $C \subset B_1(0)$ be a closed set such that $0 < |C| < e_n$ and $f: B_1(0) \setminus C \rightarrow [0,\infty)$ be a measurable function. We let $r_0 \in (0,1)$ be such that $|C| = e_n r_0^n$. We now define the \emph{ annular symmetric decreasing rearrangement}  $f^* : B_1(0) \setminus \overline{B_{r_0}(0)} \rightarrow [0,\infty)$ via its superlevel sets, demanding that for all $t \geq 0$, 
\begin{equation*}
    \{ x \in B_1(0) \setminus \overline{B_{r_0}(0)} : f^*(x) > t \} = B_{r(t)}(0) \setminus \overline{B_{r_0}(0)}, 
\end{equation*}
is satisfied, where $r(t) \in [r_0,1]$ is the unique solution of 
\begin{equation*}
    e_n r(t)^n - e_n r_0^n = |\{ x \in B_1(0) \setminus C : f(x) > t \}|. 
\end{equation*}
{From} this we conclude that for all $t \geq 0$
\begin{equation*}
     |\{ x \in B_1(0) \setminus \overline{B_{r_0}(0)} : f^*(x) > t \}| =  |\{ x \in B_1(0) \setminus C : f(x) > t \}|.
\end{equation*}
This yields for all $t\ge 0$ and $x\in B_1(0)\setminus \overline{B_{r_0}(0)}$:
\begin{equation}\label{eq:fstar_radius}
|x| < r(t)\quad \Leftrightarrow\quad e_n \left( |x|^n -r_0^n\right) < e_n \left( r(t)^n -r_0^n\right)
= |\{ y \in B_1(0) \setminus C : f(y) > t \}| .
\end{equation}
More compactly one can write, making use of \eqref{eq:fdarst}:
\begin{equation}\label{eq:fstarformula}
    f^*(x) = \int_0^\infty 1_{B_{r(t)}(0) \setminus \overline{B_{r_0}(0)}}(x) \; \mathrm{d}t, \quad \textrm{where} \quad r(t) = \left( \frac{|\{ f > t \}|}{e_n} + r_0^n\right)^\frac{1}{n}.
\end{equation}
Combining this with \eqref{eq:fstar_radius} we find further 
for $x\in B_1(0)\setminus  \overline{B_{r_0}(0)}$ that
\begin{align*}
 f^*(x) = &\int_0^\infty 1_{B_{r(t)}(0) \setminus \overline{B_{r_0}(0)}}(x) \; \mathrm{d}t
 =|\{t\in (0,\infty): r(t)>|x|\}|\\
 =& \sup \{t\in [0,\infty): r(t)>|x|\}\\
 =&\sup \left\{t\in [0,\infty):    |\{ y \in B_1(0) \setminus C : f(y) > t \}| > e_n \left( |x|^n -r_0^n\right)\right\},
\end{align*}
which may have served as an equivalent definition of 
 $f^* : B_1(0) \setminus \overline{B_{r_0}(0)} \rightarrow [0,\infty)$.

\begin{example}\label{ex:charfuncstar}
For a set $E \subset B_1(0) \setminus C$ we define 
\begin{equation}\label{eq:chiEstar}
    E^* := B_s(0) \setminus \overline{B_{r_0}(0)} \quad \textrm{where } s = \left( \frac{|E|}{e_n} + r_0^n \right)^\frac{1}{n}. 
\end{equation}
One readily checks that then $|E^*| = |E|$. We claim that  
\begin{equation}\label{eq:chiEstar2}
    (1_E)^* = 1_{E^*}. 
\end{equation}
To this end we observe that for all $x \in B_1(0) \setminus \overline{B_{r_0}(0)}$ one has 
\begin{equation}\label{eq:chiE}
    (1_E)^*(x) = \int_0^\infty 1_{B_{r(t)}(0) \setminus \overline{B_{r_0}(0)}}(x) \; \mathrm{d}t,
\end{equation}
where we have for $t \geq 0$:
\begin{equation*}
    e_n r(t)^n- e_n r_0^n  = |\{ x \in B_1(0) \setminus C : 1_E(x) > t \}|= \begin{cases}
        |E| & t \in [0,1), \\ 0 & t \geq 1.
    \end{cases}
\end{equation*}
This is equivalent to
\begin{equation*}
    r(t) = \begin{cases} \left( \frac{|E|}{e_n}+ r_0^n \right)^\frac{1}{n} & t \in [0,1), \\ r_0 & t \geq 1. \end{cases}
\end{equation*}
Using this and \eqref{eq:chiE} one readily infers  
\eqref{eq:chiEstar2}. 
\end{example}

\begin{remark}\label{rem:nurcharfunc}
The previous example and \eqref{eq:fstarformula} show  that for each measurable $f : B_1(0) \setminus C \rightarrow [0,\infty)$ one has
\begin{equation*}
    f^*(x) = \int_0^\infty (1_{\{f > t \}})^*(x) \; \mathrm{d}t. 
\end{equation*}
\end{remark}
In order to prove Theorem~\ref{thm:symmetry} we need a \emph{rearrangement inequality} for annular symmetric decreasing rearrangements. The proof is completely anlogous to the classical case of rearrangements in balls, cf. \cite[Theorem 3.4]{Lieb}. We present it for the sake of the reader's convenience. 

\begin{lemma}[Rearrangement inequality]\label{lem:rearrineq}
Let $f, g: B_1(0) \setminus C \rightarrow [0, \infty)$ be measurable functions. Then 
\begin{itemize}
	\item[(a)]
\begin{equation*}
\int_{B_1(0) \setminus C} f^2 \; \mathrm{d}x = \int_{B_1(0) \setminus \overline{B_{r_0}(0)}} (f^*)^2  \; \mathrm{d}x;
\end{equation*}
\item[(b)]
\begin{equation*}
    \int_{B_1(0) \setminus C} f g \; \mathrm{d}x \leq \int_{B_1(0) \setminus \overline{B_{r_0}(0)}} f^* g^* \; \mathrm{d}x.
 \end{equation*}
\end{itemize}
\end{lemma}
\begin{proof}
	\begin{itemize}
		\item[(a)] This claim follows from the equimeasurability of the level sets of $f$ and $f^*$ and Fubini's theorem:
		\begin{align}
	& \int_{B_1(0) \setminus C} |f|^2 \; \mathrm{d}x = \int_{B_1(0) \setminus C }\int_0^{|f(x)|} 2 t \; \mathrm{d}t \; \mathrm{d}x \nonumber \\ 
		& = \int_0^\infty 2t |\{ x \in B_1(0) \setminus C: |f | > t \}| \;  \mathrm{d}t = \int_0^\infty 2t |\{ x \in   B_1(0) \setminus \overline{B_{r_0}(0)}: f^* > t \}| \; \mathrm{d}t \nonumber	\\ 
	& = \int_{B_1(0) \setminus \overline{B_{r_0}(0)}} f^*(x)^2 \; \mathrm{d}x. \nonumber
	\end{align}
	
	\item[(b)]	
    We first assume that $f = 1_E$ and $g = 1_F$ for some measurable sets $E,F \subset B_1(0) \setminus C$. 
    Without loss of generality we may assume $|E| \geq |F|$. Using \eqref{eq:chiEstar} we infer that then $E^* \supset F^*$. This together with \eqref{eq:chiEstar2} from
    Example \ref{ex:charfuncstar} implies that
    \begin{equation*}
          (1_E)^* (1_F)^* = 1_{E^*} 1_{F^*} = 1_{E^*}.
    \end{equation*}
    Integrating we find 
    \begin{equation}\label{eq:firststepchar}
        \int_{B_1(0) \setminus \overline{B_{r_0}(0)}} (1_E)^* (1_F)^* \; \mathrm{d}x = |E^*| = |E| \geq |E \cap F| = \int_{B_1(0) \setminus C}  (1_E)(1_F) \; \mathrm{d}x.
    \end{equation}
    The claim follows therefore for characteristic functions. 
    Now let $f, g : B_1(0) \setminus C \rightarrow [0, \infty)$ be arbitrary. Using \eqref{eq:fdarst}, \eqref{eq:firststepchar} and Remark \ref{rem:nurcharfunc} we find
    \begin{align*}
        \int_{B_1(0) \setminus C} f g \; \mathrm{d}x  & = \int_{B_1(0) \setminus C} \left( \int_0^\infty 1_{ \{f > t \} }(x)  \; \mathrm{d}t \right)  \left( \int_0^\infty 1_{ \{g > s \} }(x)  \; \mathrm{d}s \right)  \; \mathrm{d}x \\ & = \int_0^\infty \int_0^\infty \int_{B_1(0) \setminus C} 1_{ \{f > t \} }(x) 1_{ \{g > s \} }(x) \; \mathrm{d}x \; \mathrm{d}t \; \mathrm{d}s 
        \\ & \leq \int_0^\infty \int_0^\infty \int_{B_1(0) \setminus \overline{B_{r_0}(0)}} (1_{ \{f > t \} })^*(x) (1_{ \{g > s \} })^*(x) \; \mathrm{d}x \; \mathrm{d}t \; \mathrm{d}s
        \\ & = \int_{B_1(0) \setminus \overline{B_{r_0}(0)}}\left( \int_0^\infty (1_{ \{f > t \} })^*(x)  \; \mathrm{d}t \right) \left( \int_0^\infty (1_{ \{g > s \} })^*(x)  \; \mathrm{d}s \right) \; \mathrm{d}x  \\ & = \int_{B_1(0) \setminus \overline{B_{r_0}(0)}} f^*(x) g^*(x) \; \mathrm{d}x.
    \end{align*}
\end{itemize}
\end{proof}


\section{Stampacchia's lemma for second derivatives}

In this appendix we prove an analogue to \emph{Stampacchia's lemma}, cf. \cite[Chapter II, Lemma A.4]{KlStam} for second derivatives, which will prove helpful for the argument. 

\begin{lemma}\label{lem:Stam}
Let $\Omega \subset \mathbb{R}^n$ and $u \in H^2(\Omega)$. Then $D^2u = 0$ a.e. on $\{ u = 0\}$.
\end{lemma}
\begin{proof}
    Let $u \in H^2(\Omega)$. By \cite[Chapter II, Lemma A.4]{KlStam} one has  that $\nabla u = 0$ a.e. on $\{ u = 0 \}$. In particular there exists some null set $N \subset \Omega$ such that $\{ u = 0 \} \subset \{ \nabla u = 0 \} \cup N$. Now notice that for each $j= 1,...,n$ one has that $\partial_j u \in H^1(\Omega)$. Hence \cite[Chapter II, Lemma A.4]{KlStam} yields that 
    \begin{equation*}
        \{ \partial_j u = 0 \} \subset \{ \nabla \partial_j u = 0 \} \cup N_j \quad \forall j = 1,...,n
    \end{equation*}
    for a null set $N_j$. Defining $M := \bigcup_{j = 1}^n N_j$ we obtain 
    \begin{align*}
        \{ u = 0 \} & \subset \{ \nabla u = 0 \} \cup N \subset \bigcap_{j = 1}^n \{ \partial_j u =0 \} \cup N 
        \subset \bigcap_{j = 1}^n (\{ \nabla \partial_j u =0 \}  \cup N_j) \cup  N  \\ &\subset \bigcap_{j = 1}^n (\{ \nabla \partial_j u =0 \} \cup M \cup N  = \{ D^2 u = 0\} \cup (M \cup N).
    \end{align*}
    Since $M \cup N$ is a null set we infer that $D^2u = 0$ a.e. on $\{ u = 0 \}$.
\end{proof}

\bigskip\noindent
{\bfseries Acknowledgement.} 
We are grateful to the referees for a very careful reading and for giving a very detailed and extremely valuable feedback on our submission. This led not only to improving the exposition but also to correcting two errors and some inaccuracies from the previous version.

The second author is also grateful to Mickael Nahon for helpful discussions.


\end{document}